%% file: real_numbers.tex
\documentclass{article}

\usepackage{lineno,hyperref}
\modulolinenumbers[5]

\usepackage{authblk}

\usepackage{algorithm}
\usepackage[noend]{algpseudocode}

\usepackage{subfigure}
\usepackage{graphicx}
\usepackage{amssymb}
\usepackage{amsmath}
\usepackage{pgfplots}
\usepackage{tikz}
\usetikzlibrary{shapes,arrows}
\usepackage{paralist}
\usepackage{amsthm}
\usepackage{varwidth}
\usepackage{caption}

\usepackage{enumerate}

\makeatletter
\def\BState{\State\hskip-\ALG@thistlm}
\makeatother

\newcommand{\true}{\textit{true}}
\newcommand{\expr}{\textit{expr}}
\newcommand{\nat}{\mathbb{N}}

\newcommand{\nonzerointeger}{\mathbb{Z}\setminus \{0\}}
\newcommand{\integer}{\mathbb{Z}}
\newcommand{\real}{\mathbb{R}}

\newcommand{\dx}{\Delta a}
\newcommand{\dy}{\Delta b}
\newcommand{\sqrtexp}{\log_2(\log_2(2^{p+3}+1))}
\newcommand{\om}{\overline{m}}
\newcommand{\on}{\overline{n}}

\algnewcommand{\algorithmicgoto}{\textbf{go to}}
\algnewcommand{\Goto}[1]{\algorithmicgoto~\ref{#1}}
\algnewcommand{\LineComment}[1]{\State \(\triangleright\) #1}

\newtheorem{mytheorem}{Theorem}
\numberwithin{mytheorem}{section}

\newtheorem{definition}{Definition}
\newtheorem{proposition}{Proposition}
\newtheorem{lemma}{Lemma}

\newcommand{\taylorsin}[3] {\sum_{i=#1}^{#2}\frac{(-1)^i}{(2i+1)!}{#3}^{2i+1}}

\newcommand{\taylorcos}[3] {\sum_{i=#1}^{#2}\frac{(-1)^i}{(2i)!}{#3}^{2i}}

\newcommand{\taylorarctan}[3] {\sum_{i=#1}^{#2}\frac{(-1)^i}{2i+1                                                                                                                                                                                                                                                                                                                                                                                                                                                                       }{#3}^{2i+1}}

\newcommand{\lnsumu}[2] {\frac{#1-1}{#2} \sum_{i=0}^{#2-1}\frac{1}{1+\frac{i}{#2}(#1-1)}}

\newcommand{\lnsuml}[2] {\frac{#1-1}{#2} \sum_{i=1}^{#2}\frac{1}{1+\frac{i}{#2}(#1-1)}}

\newcommand{\arctansumu}[2] {\frac{#1}{#2} \sum_{i=0}^{#2-1}\frac{1}{1+(\frac{i}{#2})^2{#1}^2}}

\newcommand{\arctansuml}[2] {\frac{#1}{#2} \sum_{i=1}^{#2}\frac{1}{1+(\frac{i}{#2})^2{#1}^2}}

\DeclareMathOperator{\arccot}{arccot}

\begin{document}
\tikzstyle{block} = [draw, fill=white, rectangle, 
    minimum height=3em, minimum width=6em]
\tikzstyle{sum} = [draw, fill=white, circle, node distance=1cm]
\tikzstyle{input} = [coordinate]
\tikzstyle{output} = [coordinate]
\tikzstyle{pinstyle} = [pin edge={to-,thin,black}]

\title{Exact Real Arithmetic with \\
Perturbation Analysis and Proof of Correctness}

\author{Sarmen Keshishzadeh}
\author{Jan Friso Groote}
\affil{Department of Mathematics and Computer Science\\
Eindhoven University of Technology\\
Den Dolech 2, 5612 AZ Eindhoven, The Netherlands}

%\address{Department of Mathematics and Computer Science, Eindhoven University of Technology, Den Dolech 2, 5612 AZ Eindhoven, The Netherlands}

%\cortext[mycorrespondingauthor]{Corresponding author}{\ead{s.keshishzadeh@tue.nl, j.f.groote@tue.nl}}

\maketitle

\begin{abstract}
In this article, we consider a simple representation for real numbers and propose top-down procedures to approximate various algebraic and transcendental operations with arbitrary precision. Detailed algorithms and proofs are provided to guarantee the correctness of the approximations. Moreover, we develop and apply a perturbation analysis method to show that our approximation procedures only recompute expressions when unavoidable. 

In the last decade, various theories have been developed and implemented to realize real computations with arbitrary precision. Proof of correctness for existing approaches typically consider basic algebraic operations, whereas detailed arguments about transcendental operations are not available. Another important observation is that in each approach some expressions might require iterative computations to guarantee the desired precision. However, no formal reasoning is provided to prove that such iterative calculations are essential in the approximation procedures. In our approximations of real functions, we explicitly relate the precision of the inputs to the guaranteed precision of the output, provide full proofs and a precise analysis of the necessity of iterations. 
\end{abstract}

%\begin{keyword}
%Exact Real Arithmetic, Perturbation Analysis, Approximation by Riemann Sums, Approximation by Taylor Expansions, Proof of Correctness
%\end{keyword}

\input{introduction}

\input{computability}

\input{representation}

\input{conditioning}

\input{algebraic_operations}

\input{transcendental_func_riemann}

\input{transcendental_func_taylor}

\input{related_work}

\input{conclusion}

\paragraph{Acknowledgement}
This research was supported by the Dutch national program
COMMIT and carried out as part of the Allegio project.

\bibliography{real_numbers}

\appendix

\input{useful_facts}

\end{document}

%% file: introduction.tex
\section{Introduction}
\label{sec:intro}
Various scientific disciplines use computations involving real numbers to model and reason about different phenomena in the world. Real numbers are typically approximated by floating point numbers in scientific calculations. Round-off errors are inevitable in such approximations and they might build up into catastrophic errors in some cases. Exact real arithmetic approaches address this issue by devising computation procedures that given an expression and a precision requested by the user produce an output that is guaranteed to meet the precision requirement. 

Several approaches \cite{M01,MPFR} to exact real arithmetic are based on iterative bottom-up calculations. Given an expression and a desired precision, bottom-up approaches typically start with calculating the inputs with an arbitrary precision higher than the requested precision. Then, the sub-expressions are evaluated in a bottom-up way. After evaluating the sub-expressions of each level, the guaranteed precision is passed to the higher level. These calculations proceed until the main expression is calculated and its guaranteed precision is determined. If the precision obtained for the expression is not adequate, the computation restarts with increased precisions for the inputs. 

In contrast, various authors \cite{GL00,M05} have proposed top-down approaches to exact real arithmetic. In top-down approaches, the required precision of each sub-expression is determined based on the precision required for its immediate parent expression. For certain types of expressions, the required precision of the sub-expressions can be calculated immediately. However, some expressions may require to first obtain additional information about the magnitude of the values of their sub-expressions before determining their required precision. Thus, in general, it might be necessary to recompute certain expressions. 

%Various implementations for the mentioned approaches have been proposed. Currently, increasing attention is being given to making such computation systems suitable for practical purposes or to extend them with more advanced features. The results of a recent competition \cite{B01} between a few available implementations indicate that some implementations are capable of computing non-trivial expressions in a reasonable amount of time.  

%Existing exact real arithmetic approaches have explored various representations for real numbers and proposed approximations for algebraic operations (e.g., addition) and transcendental functions (e.g., natural logarithm). Studying the literature, we have identified two important aspects that are overlooked in existing approaches.

The main benefit of top-down approaches is that they exploit the structure of a given expression to estimate the required precision of the sub-expressions. In this context, one would ideally like to have top-down approximations for algebraic and transcendental functions such that 
\begin{inparaenum}[1)]
\item
the approximations are proven to be correct and 
\item
iterative calculations are avoided unless they are proven to be necessary.
\end{inparaenum}

In several studies \cite{GL00,O08}, proofs of correctness for algebraic operations are available. However, the arguments about transcendental functions provide little insight about the correctness of the approximations and the effect of these operations on precision. Taylor expansions are the most prominent way to approximate transcendental functions. Calculations with Taylor expansions are typically restricted to a base interval; range reduction identities are used to extend the computations to the complete domain of a function. Proofs of correctness for transcendental functions are limited to the base interval \cite{P98,GL00,O08}, whereas little attention is given to the general case where the computations introduced by range reduction identities influence the output precision.  

The second desired property for a top-down approach is related to the iterative nature of the computations. As discussed above, bottom-up and top-down approaches rely on iterative computation schemes. However, no formal reasoning is provided to prove that such iterative calculations are essential in the approximation procedures.  

In this article, we consider a simple representation for real numbers and propose a top-down approach to approximate various algebraic and transcendental functions with arbitrary precision. For each operation, we describe an approximation procedure and relate the precision of the inputs to the guaranteed precision for the output. To guarantee the correctness of our top-down approach, we provide detailed proofs of correctness for the proposed approximations. 

To identify computational problems that require iterative calculations in our top-down approach, we have developed a perturbation analysis method. Our analysis describes the influence of errors in the inputs of a computational problem on the output precision. We apply perturbation analysis to show that our approximation procedures only recompute expressions when this is unavoidable. 

\paragraph{Overview}
We discuss different approaches to defining computability of functions in Section~\ref{sec:computability}. In Section~\ref{sec:representation} we introduce a representation for real numbers and specify the syntax of the expressions that we consider in our computations. To analyze computational problems in this setting, a perturbation analysis method is introduced in Section~\ref{sec:condition}. In Section~\ref{sec:algebraic_operations} we discuss our approximations of algebraic operations. We approximate transcendental functions using Riemann sums and Taylor expansions in Section~\ref{sec:riemann} and \ref{sec:taylor}, respectively. Section~\ref{sec:related_work} contains discussions about related work. In Section~\ref{sec:conclusion} we draw some conclusions and suggest directions for future research.

%% file: computability.tex
\section{Computable Real Functions}
\label{sec:computability}
Real arithmetic is concerned with performing computations on  real numbers. In order to do calculations with real numbers, it is necessary to define what it means for an  operation to be computable. In this section we briefly discuss different approaches to defining computability. %In the following discussion, we assume that $\Sigma$ is a finite alphabet. 

Since real numbers are infinite objects, one can use infinite streams from a finite alphabet $\Sigma$ to represent them. This gives rise to a definition of computability called Type-$2$ Theory of Effectivity (TTE, \cite{W12}). In TTE computable operations are defined in terms of functions $f:\Sigma^{\omega}\rightarrow \Sigma^{\omega}$ that receive infinite words as input and produce infinite words as output. An essential property of a Type-$2$ computable function is the finiteness property \cite{W12}. This property indicates that for a computable function $f$, any finite prefix of the output $f(x)$ is computable by some finite portion of the input $x$.

An alternative definition of computability has been introduced by the Russian school of constructive analysis \cite{M54,K84}. In this definition, computable operations are defined in terms of functions $f:\Sigma^*\rightarrow \Sigma^*$ that receive finite words as input and produce finite words as output. This approach is sometimes referred to as Type-$1$ computability. A function $f$ is Type-$1$ computable if there is a Turing machine that transforms any finite input string $x\in \Sigma^*$ to the finite output strings $f(x)\in \Sigma^*$. Type-$1$ machines provide a natural way to define computability on, for instance, rational numbers and finite graphs.

In both Type-$1$ and Type-$2$ approaches, $\Sigma$ depends on the concrete representation that we use for input/output objects. For instance, one can use the binary signed-digit representation to represent real numbers in the inputs/outputs of computations. 

The relation between Type-$1$ and Type-$2$ computable functions has been investigated in various studies. It is known that restricting the domain of a Type-$2$ computable function to finite streams results in a Type-$1$ computable function \cite{H96,RW99}. However, not every Type-$1$ function can be obtained by restricting some Type-$2$ computable function \cite{RW99,KLS00}. To illustrate this, we consider a function $f_1: \{0,1\}^*\rightarrow \{0,1\}^*$ defined as follows:
\begin{align*}
f_1(s)=
\begin{cases}
0& \text{if } s=0^k\\
1& \text{if } s=0^k1s'
\end{cases}
\end{align*}
where $k\in \nat$ is the length of the longest prefix of zeros in $s$ and $s'\in \Sigma^*$ is a finite suffix of $s$ in the second case. The function $f_1$ performs computations on finite strings and one can construct a Type-$1$ Turing machine to compute this function. By extending the domain of $f_1$ to infinite strings, we obtain $f_2: \Sigma^{\omega} \rightarrow \Sigma^{\omega}$ such that:
\begin{align*}
f_2(s)=
\begin{cases}
0& \text{if } s=0^\omega\\
1& \text{if } s=0^k1s'
\end{cases}
\end{align*}
The function $f_2$ is not computable with a Type-$2$ Turing machine; it is not possible to write $0$ in the output after reading a finite prefix from the input. The interested reader can refer to \cite{H96} for more details about the relation between Type-$1$ and Type-$2$ computable functions. 

In addition to Type-$1$ and Type-$2$, one can also consider a third approach to defining computability based on certain finite structures that provide precise descriptions for specific classes of real numbers. For instance, Lagrange's theorem on continued fractions indicates that the real numbers whose continued fraction is periodic are the quadratic irrationals. Based on this observation, one can define computability in terms of functions $f$ that given a finite and precise representation of $x$ produce a finite and precise representation of $f(x)$.   

The exact real arithmetic approach that we introduce in this article is based on Type-$2$ computability. In Section~\ref{sec:representation} we discuss a representation for real numbers in terms of rational numbers that are coupled with a notion of precision. Our approximations for arithmetic operations rely on the finiteness property of computable functions. Thus, for a given computational problem, a desired precision for the output is obtained based on sufficiently good approximations of the inputs. 

%% file: representation.tex
\section{Real Numbers: Representation \& Operations}
\label{sec:representation}
In this section we first discuss our representation of real numbers and then describe the syntax of the expressions that we aim to calculate in our setting. 

Since real numbers are infinite objects, a finite representation of an arbitrary real number $x$ can only approximate $x$ with a certain precision. In scientific measurements and calculations, the amount of error that we commit in approximations is measured by an absolute or relative error. In practice, an absolute error is of little use. Since numbers tend to have very different orders of magnitude, it is the relative error that shows the significance of the lost digits in measurements or calculations. Hence, in our setting we use a representation based on the relative error.
\begin{definition}\label{def:realrep}
A real number $x$ is represented by a tuple $(m,n,p)$ such that:
\begin{align*}
|x-\frac{m}{n}| < |\frac{m}{n}|\frac{1}{2^p}
\end{align*}
where $m,n\in \nonzerointeger, p\in \nat$.  
\end{definition}
The representation $(m,n,p)$ for $x$ means that $\frac{m}{n}$ approximates $x$ and the relative error of this approximation does not exceed $\frac{1}{2^p}$.

In this article, we focus on calculating expressions that can be described with the following grammar:
\begin{equation}
\begin{aligned}
E~::=&
	~c~|
 	~-E~|
 	~E\cdot E~|
 	~\frac{1}{E}~|
 	~E+E~|
 	~\sqrt{E}~|
 	~e^E~|\\
 &	~\ln(E)~|
 	~\arctan(E)~|
 	~\cos(E)~|
 	~\sin(E)\label{eqn:grammar}
\end{aligned}
\end{equation}
where $c$ represents a rational constant. 

The following identities show that other interesting operations can be described in terms of the operations of this grammar:
\begin{align*}
\tan(x)&=\frac{\sin(x)}{\cos(x)}\\
\cot(x)&=\frac{1}{\tan(x)}\\
\arcsin(x)&=\arctan(\frac{x}{\sqrt{1-x^2}})\\
\arccos(x)&=\arctan(\frac{\sqrt{1-x^2}}{x})\\
\arccot(x)&=\arccos(\frac{x}{\sqrt{1+x^2}})
\end{align*}

%% file: conditioning.tex
\section{Sensitivity of Operations to Perturbations in the Arguments}\label{sec:condition}
Our goal is to develop a top-down exact real arithmetic approach based on the representation of Definition~\ref{def:realrep}. Thus, for a given computational problem, it is essential to estimate the required precision of the inputs based on the desired precision in the output. Moreover, we would like to investigate to what extend the operations of  grammar~\eqref{eqn:grammar} can be calculated in a top-down manner without iterations. 

To analyze the operations of grammar~\eqref{eqn:grammar}, we introduce a pertubation analysis method for measuring the sensitivity of the operations to perturbations in their arguments. We consider two general cases in our analysis. First, we consider a function $f(x)$ in one variable and show how errors in the input influence the output (Section~\ref{subsec:pert_unary}). Then, we consider a function $f(x,y)$ with two arguments and investigate the effect of errors in the inputs on the output (Section~\ref{subsec:pert_bin}).

\subsection{Perturbation Analysis for Unary Functions}\label{subsec:pert_unary}
Let $f(x)$ be a differentiable function that we want to calculate in point $x=a$. Suppose that $\dx$ is a perturbation in the argument $a$. The relative error in the calculation of $f(a)$ caused by this perturbation is:
\begin{align*}
|\frac{f(a+\dx)-f(a)}{f(a)}|
\end{align*}
We want to relate this relative error to the relative error of the argument, namely $|\frac{\dx}{a}|$. To this end, we use the following approximation of the function $f$ in point $x=a+\dx$:
\begin{align*}
f(a+\dx)\approx f(a)+f'(a)\dx
\end{align*}
We can approximate the relative error of $f$ as follows:
\begin{align}
|\frac{f(a+\dx)-f(a)}{f(a)}|\approx |\frac{f'(a)\dx}{f(a)}|=|\frac{af'(a)}{f(a)}||\frac{\dx}{a}|\label{eqn:condition_single}
\end{align}
From equality~\eqref{eqn:condition_single} one can see that the quantity $|\frac{af'(a)}{f(a)}|$ determines the effect of the relative error $|\frac{\dx}{a}|$ on the output. In numerical analysis and linear algebra the quantity $|\frac{xf'(x)}{f(x)}|$ is usually referred to as the \textit{condition number} of $f(x)$ \cite{H02,H08}.

\subsection{Perturbation Analysis for Binary Functions}\label{subsec:pert_bin}
Let $f(x,y)$ be a differentiable function that we want to calculate in point $(x,y)=(a,b)$. Suppose $\dx$ and $\dy$ are perturbations in the arguments $a$ and $b$, respectively. The relative error of $f(a,b)$ caused by these perturbations can be calculated as follows:
\begin{align}
|\frac{f(a+\dx,b+\dy)-f(a,b)}{f(a,b)}|\label{eqn:error_multiple}
\end{align}
To find an upper-bound for \eqref{eqn:error_multiple}, we use the following first-order approximation of the function $f$ in point $(x,y)=(a+\dx, b+\dy)$:
\begin{align*}
f(a+\dx,b+\dy)\approx f(a,b)+f_x(a,b)\dx+f_y(a,b)\dy
\end{align*}
where $f_x(a,b)$ and $f_y(a,b)$ are the partial derivatives of $f$ with respect to $x$ and $y$ in $(a,b)$, respectively. We relate the relative error in the calculation of $f$ to the relative errors $|\frac{\dx}{a}|$ and $|\frac{\dy}{b}|$ as follows:
\begin{align}
|\frac{f(a+\dx,b+\dy)-f(a,b)}{f(a,b)}|&\approx |\frac{f_x(a,b)\dx +f_y(a,b)\dy}{f(a,b)}|\leq \nonumber\\
& |\frac{af_x(a,b)}{f(a,b)}||\frac{\dx}{a}|+|\frac{bf_y(a,b)}{f(a,b)}||\frac{\dy}{b}|\leq \nonumber\\
& (|\frac{af_x(a,b)}{f(a,b)}|+|\frac{bf_y(a,b)}{f(a,b)}|)\max\{|\frac{\dx}{a}|,|\frac{\dy}{b}|\}\label{eqn:condition_multiple}
\end{align}   
The quantity $|\frac{af_x(a,b)}{f(a,b)}|+|\frac{bf_y(a,b)}{f(a,b)}|$ determines the upper-bound calculated in inequality~\eqref{eqn:condition_multiple} and we use this quantity to measure the effect of erroneous arguments on the output. It should be noted that in inequality~\eqref{eqn:condition_multiple} we have considered $|\frac{\dx}{a}|$ and $|\frac{\dy}{b}|$ as independent factors that can influence the relative error of $f$. This way of reasoning about the sensitivity of $f(x,y)$ is related to \textit{componentwise} analysis of perturbation in numerical analysis and linear algebra \cite{HD12}, which we use in this article. 

Another possibility is to relate the relative error of \eqref{eqn:error_multiple} to the quantity:
\begin{align*}
\frac{
	\begin{Vmatrix}
 		\begin{bmatrix}
			\dx \\ \dy
		\end{bmatrix}
	\end{Vmatrix}  }
   {\begin{Vmatrix}   
		\begin{bmatrix}
			a \\ b
		\end{bmatrix} 
	\end{Vmatrix}}
\end{align*}
This type of analysis is usually referred to as \textit{normwise} analysis of perturbation~\cite{HD12}. 

In the following sections, we provide a top-down approach for approximating various algebraic and transcendental functions. Perturbation analysis will be used to show that  in our approximations we only recompute expressions when essential. 

%% file: algebraic_operations.tex
\section{Approximating Algebraic Operations}
\label{sec:algebraic_operations}
In this section we calculate the algebraic operations of grammar~\eqref{eqn:grammar} using a top-down approach. We formulate and prove theorems that allow us to calculate expressions involving unary negation (Section~\ref{subsec:neg}), multiplication (Section~\ref{subsec:mult}), inverse (Section~\ref{subsec:inv}), addition (Section~\ref{subsec:add}), and square root (Section~\ref{subsec:sqrt}). Based on the theorems, we provide different implementations of \textproc{Compute}$(expr,p)$ to calculate algebraic operations. These implementations receive an algebraic expression $expr$ and a desired precision $p$ and produce an output with the desired precision. In each case, we also apply the perturbation analysis of Section~\ref{sec:condition} and show that we avoid unnecessary iterations in our approximations.
 
\subsection{Unary Negation}
\label{subsec:neg} 
\begin{mytheorem}\label{thm:unary_neg}
Let $x$ be a real number represented by $(m,n,p)$. Then $-x$ can be represented by $(-m,n,p)$.
\end{mytheorem}
\begin{proof}
Since $x$ is represented by $(m,n,p)$ we can write: 
\begin{align*}
\frac{m}{n}-|\frac{m}{n}|\frac{1}{2^p}&<x<\frac{m}{n}+|\frac{m}{n}|\frac{1}{2^p}\\
-\frac{m}{n}-|\frac{m}{n}|\frac{1}{2^p}&<-x<-\frac{m}{n}+|\frac{m}{n}|\frac{1}{2^p}
\end{align*}
Thus, we can represent $-x$ by $(-m,n,p)$.

\end{proof}

Algorithm~\ref{alg:negation} applies Theorem~\ref{thm:unary_neg} and approximates $-x$ based on a representation $(m,n,p)$ of $x$. To confirm that $-x$ can be approximated with arbitrary precision in one pass, we calculate $|\frac{xf'(x)}{f(x)}|$ for the function $f(x)=-x$:
\begin{align*}
|\frac{xf'(x)}{f(x)}|=|\frac{(x)(-1)}{-x}|=1
\end{align*} 
The quantity $|\frac{xf'(x)}{f(x)}|$ is small and independent of the argument $x$ and hence the amount of precision that we lose in unary negation (which is $0$ in the case of Theorem~\ref{thm:unary_neg}) can be calculated independently of $x$.

\begin{algorithm}
\caption{Unary Negation}\label{alg:negation}
\begin{algorithmic}[1]
\Require $\expr$ has the shape $-x$
\Procedure{Compute}{$\expr,p$}

\State $\frac{m}{n} \gets \Call{Compute}{x,p}$
\State \Return $\frac{-m}{n}$
\EndProcedure
\end{algorithmic}
\end{algorithm} 

\subsection{Multiplication $x\cdot y$}
\label{subsec:mult}

\begin{mytheorem}\label{thm:mult}
Let $x$ and $y$ be two real numbers represented by $(m,n,p)$ and $(m',n',p)$, respectively. Then $x\cdot y$ can be represented by $(mm',nn',p-2)$. 
\end{mytheorem}
\begin{proof}
From Definition~\ref{def:realrep} we can write:
\begin{align}
\frac{m}{n}-|\frac{m}{n}|\frac{1}{2^{p}}<x<\frac{m}{n}+|\frac{m}{n}|\frac{1}{2^{p}} \label{eqn:mult_x}\\
\frac{m'}{n'}-|\frac{m'}{n'}|\frac{1}{2^{p}}<y<\frac{m'}{n'}+|\frac{m'}{n'}|\frac{1}{2^{p}} \label{eqn:mult_y}
\end{align}
We consider three cases:
\begin{enumerate}
\item
Suppose $\frac{mm'}{nn'}>0$. We can multiply inequalities~\eqref{eqn:mult_x} and \eqref{eqn:mult_y} as follows:
\begin{align*}
\frac{mm'}{nn'}(1-\frac{1}{2^{p}})^2<x\cdot y <\frac{mm'}{nn'}(1+\frac{1}{2^{p}})^2\
\end{align*}
If $x\cdot y$ can be represented by $(mm',nn',p-2)$ then it must be the case that:
\begin{align*}
\frac{mm'}{nn'}(1-\frac{1}{2^{p-2}})<x\cdot y<\frac{mm'}{nn'}(1+\frac{1}{2^{p-2}})
\end{align*}
To see that this is valid, we need to show that:
\begin{align*}
\frac{mm'}{nn'}(1+\frac{1}{2^{p}})^2\leq \frac{mm'}{nn'}(1+\frac{1}{2^{p-2}})\\
\frac{mm'}{nn'}(1-\frac{1}{2^{p-2}}) \leq \frac{mm'}{nn'}(1-\frac{1}{2^{p}})^2
\end{align*}
But $\frac{mm'}{nn'}>0$ and from Proposition~\ref{prop:prec2} (see \ref{app:prop_lem}) we know that both inequalities hold.
\item
Suppose $\frac{m}{n}>0, \frac{m'}{n'}<0$. We can rewrite inequalities~\eqref{eqn:mult_x} and \eqref{eqn:mult_y} as follows:
\begin{align}
\frac{m}{n}(1-\frac{1}{2^{p}})<x<\frac{m}{n}(1+\frac{1}{2^{p}})\label{eqn:mult_x_pos} \\
\frac{m'}{n'}(1+\frac{1}{2^{p}})<y<\frac{m'}{n'}(1-\frac{1}{2^{p}}) \label{eqn:mult_y_neg}
\end{align}
Multiplying inequalities~\eqref{eqn:mult_x_pos} and \eqref{eqn:mult_y_neg} we get:
\begin{align*}
\frac{mm'}{nn'}(1+\frac{1}{2^{p}})^2<x\cdot y<\frac{mm'}{nn'}(1-\frac{1}{2^{p}})^2
\end{align*}
If $x\cdot y$ is representable by $(mm',nn',p-2)$ then it must be the case that:
\begin{align*}
\frac{mm'}{nn'}(1+\frac{1}{2^{p-2}})<x\cdot y<\frac{mm'}{nn'}(1-\frac{1}{2^{p-2}})
\end{align*}
To show that this is valid, it suffices to prove that:
\begin{align*}
\frac{mm'}{nn'}(1-\frac{1}{2^{p}})^2 \leq \frac{mm'}{nn'}(1-\frac{1}{2^{p-2}}) \\
\frac{mm'}{nn'}(1+\frac{1}{2^{p-2}})\leq \frac{mm'}{nn'}(1+\frac{1}{2^{p}})^2
\end{align*}
But $\frac{mm'}{nn'}<0$ and from Proposition~\ref{prop:prec2} (see \ref{app:prop_lem}) we know that both inequalities hold. 
\item
Suppose $\frac{m}{n}<0, \frac{m'}{n'}>0$. This case can be proved similarly to the second case.
\end{enumerate}
\end{proof}

\begin{algorithm}
\caption{Multiplication}\label{alg:mult}
\begin{algorithmic}[1]
\Require $\expr$ has the shape $x\cdot y$
\Procedure{Compute}{$\expr,p$}

\State $\frac{m}{n} \gets \Call{Compute}{x,p+2}$
\State $\frac{m'}{n'} \gets \Call{Compute}{y,p+2}$
\State \Return $\frac{mm'}{nn'}$
\EndProcedure
\end{algorithmic}
\end{algorithm}

Algorithm~\ref{alg:mult} depicts an approximation of $x\cdot y$ based on Theorem~\ref{thm:mult}. Given approximations of $x$ and $y$, this algorithm approximates $x\cdot y$ in one pass; the loss of precision is predictable without additional information about $x$ and $y$. To confirm this claim, we apply the perturbation analysis of Section~\ref{sec:condition} on the function $f(x,y)=x\cdot y$:
\begin{align*}
|\frac{xf_x(x,y)}{f(x,y)}|+|\frac{yf_y(x,y)}{f(x,y)}|=|\frac{x.y}{x.y}|+|\frac{y.x}{x.y}|=2
\end{align*}
The sensitivity measure is small and independent of the arguments. Thus, in a top-down approach, $x\cdot y$ can be approximated in one pass without iterative computations of $x$ and $y$.

\subsection{Inverse}\label{subsec:inv}
\begin{mytheorem}\label{thm:inverse}
Let $x$ be a real number represented by $(m,n,p)$. Then $\frac{1}{x}$ can be represented by $(n,m,p-1)$.
\end{mytheorem}
\begin{proof}
We consider two cases:
\begin{enumerate}
\item
Suppose $\frac{m}{n}>0$. Since $x$ is represented by $(m,n,p)$ we can write:
\begin{align*}
\frac{m}{n}(1-\frac{1}{2^{p}})<x<\frac{m}{n}(1+\frac{1}{2^{p}})\\
\frac{n}{m}(\frac{2^{p}}{2^{p}+1})<\frac{1}{x}<\frac{n}{m}(\frac{2^{p}}{2^{p}-1})
\end{align*}
If $\frac{1}{x}$ is representable by $(n,m,p-1)$, then it must be the case that:
\begin{align*}
\frac{n}{m}(1-\frac{1}{2^{p-1}})<\frac{1}{x}<\frac{n}{m}(1+\frac{1}{2^{p-1}})
\end{align*}
To see that this is valid, it suffices to show:
\begin{align*}
\frac{n}{m}(\frac{2^{p}}{2^{p}-1})\leq \frac{n}{m}(1+\frac{1}{2^{p-1}})\\
\frac{n}{m}(1-\frac{1}{2^{p-1}})\leq \frac{n}{m}(\frac{2^{p}}{2^{p}+1})
\end{align*}
But $\frac{n}{m}>0$ and hence both inequalities follow from Proposition~\ref{prop:prec1} (see \ref{app:prop_lem}). 
\item
Suppose $\frac{m}{n}<0$. Since $x$ is representable by $(m,n,p)$ we have:
\begin{align*}
\frac{m}{n}(1+\frac{1}{2^{p}})<x<\frac{m}{n}(1-\frac{1}{2^{p}})\\
\frac{n}{m}(\frac{2^{p}}{2^{p}-1})<\frac{1}{x}<\frac{n}{m}(\frac{2^{p}}{2^{p}+1})
\end{align*}
If $\frac{1}{x}$ is representable by $(n,m,p-1)$ then it must be the case that:
\begin{align*}
\frac{n}{m}(1+\frac{1}{2^{p-1}})<\frac{1}{x}<\frac{n}{m}(1-\frac{1}{2^{p-1}})
\end{align*}
To see that this is valid, we need to show:
\begin{align*}
\frac{n}{m}(\frac{2^{p}}{2^{p}+1})\leq \frac{n}{m}(1-\frac{1}{2^{p-1}}) \\
\frac{n}{m}(1+\frac{1}{2^{p-1}})\leq \frac{n}{m}(\frac{2^{p}}{2^{p}-1})
\end{align*}
But $\frac{n}{m}<0$ and hence the inequalities follow from Proposition~\ref{prop:prec1} (see \ref{app:prop_lem}). 
\end{enumerate}
\end{proof}

\begin{algorithm}
\caption{Inverse}\label{alg:inverse}
\begin{algorithmic}[1]
\Require $expr$ has the shape $\frac{1}{x}$
\Procedure{Compute}{$expr,p$}

\State $\frac{m}{n} \gets \Call{Compute}{x,p+1}$
\State \Return $\frac{n}{m}$
\EndProcedure
\end{algorithmic}
\end{algorithm}

Algorithm~\ref{alg:inverse} approximates $\frac{1}{x}$ with precision $p$ based on Theorem~\ref{thm:inverse}. Given an approximation of $x$ with precision $p$, the algorithm allows us to approximate $\frac{1}{x}$ in one pass. We use perturbation analysis and calculate the quantity $|\frac{xf'(x)}{f(x)}|$ for $f(x)=\frac{1}{x}$ to show that loss of precision in the inverse can be estimated independently of the argument: 
\begin{align*}
|\frac{xf'(x)}{f(x)}|=|\frac{(x)(\frac{-1}{x^2})}{(\frac{1}{x})}|=1
\end{align*}
The quantity $|\frac{xf'(x)}{f(x)}|$ is a constant and hence iterative computations can be avoided when calculating the inverse.

\subsection{Addition}\label{subsec:add}   
\begin{mytheorem}\label{thm:add}
Let $x$ and $y$ be two real numbers represented by $(m,n,p)$ and $(m',n',p)$, respectively. The value of $x+y$ can be approximated as follows:
\begin{enumerate}[i.]
\item \label{thm:add_same_sign}
If $\frac{mm'}{nn'}>0$, then $x+y$ can be represented by $(mn'+m'n,nn',p)$.
\item \label{thm:add_diff_sign}
If $\frac{mm'}{nn'}<0$ and $i\in \nat^+$ is the smallest natural number such that $i\geq \log_2(\frac{1+\frac{\min(|\frac{m}{n}|,|\frac{m'}{n'}|)}{\max(|\frac{m}{n}|,|\frac{m'}{n'}|)}}{1-\frac{\min(|\frac{m}{n}|,|\frac{m'}{n'}|)}{\max(|\frac{m}{n}|,|\frac{m'}{n'}|)}})$, then $x+y$ can be represented by $(mn'+m'n,nn',p-i)$.
\end{enumerate}
\end{mytheorem}

\begin{proof}
\textit{\\}
\begin{enumerate}[i.]
\item
For numbers $x$ and $y$ we can write:
\begin{align}
\frac{m}{n}-|\frac{m}{n}|\frac{1}{2^{p}}&<x<\frac{m}{n}+|\frac{m}{n}|\frac{1}{2^{p}} \label{eqn:add_x}\\
\frac{m'}{n'}-|\frac{m'}{n'}|\frac{1}{2^{p}}&<y<\frac{m'}{n'}+|\frac{m'}{n'}|\frac{1}{2^{p}} \label{eqn:add_y}
\end{align}
From inequalities~\eqref{eqn:add_x} and \eqref{eqn:add_y} we can write:
\begin{align}
(\frac{m}{n}+\frac{m'}{n'})-\frac{1}{2^{p}}(|\frac{m}{n}|+|\frac{m'}{n'}|)<x+y<(\frac{m}{n}+\frac{m'}{n'})+\frac{1}{2^{p}}(|\frac{m}{n}|+|\frac{m'}{n'}|)\label{eqn:add}
\end{align}
If $x+y$ is representable by $(mn'+m'n,nn',p)$, then it must be the case that:
\begin{align*}
(\frac{m}{n}+\frac{m'}{n'})-\frac{1}{2^{p}}|\frac{m}{n}+\frac{m'}{n'}|<x+y<(\frac{m}{n}+\frac{m'}{n'})+\frac{1}{2^{p}}|\frac{m}{n}+\frac{m'}{n'}|
\end{align*}
To show that this is valid, we need to prove that:
\begin{align}
(\frac{m}{n}+\frac{m'}{n'})+\frac{1}{2^{p}}(|\frac{m}{n}|+|\frac{m'}{n'}|)\leq (\frac{m}{n}+\frac{m'}{n'})+\frac{1}{2^{p}}|\frac{m}{n}+\frac{m'}{n'}| \label{eqn:add_samesign_pos} \\
(\frac{m}{n}+\frac{m'}{n'})-\frac{1}{2^{p}}|\frac{m}{n}+\frac{m'}{n'}| \leq (\frac{m}{n}+\frac{m'}{n'})-\frac{1}{2^{p}}(|\frac{m}{n}|+|\frac{m'}{n'}|) \label{eqn:add_samesign_neg}
\end{align} 
The rational numbers $\frac{m}{n}$ and $\frac{m'}{n'}$ have the same sign. Therefore, $|\frac{m}{n}+\frac{m'}{n'}|=|\frac{m}{n}|+|\frac{m'}{n'}|$ holds and inequalities~\eqref{eqn:add_samesign_pos} and \eqref{eqn:add_samesign_neg} are valid. 
\item
If $x+y$ is representable by $(mn'+m'n,nn',p-i)$, then it must be the case that:
\begin{align*}
(\frac{m}{n}+\frac{m'}{n'})-\frac{1}{2^{p-i}}|\frac{m}{n}+\frac{m'}{n'}|<x+y<(\frac{m}{n}+\frac{m'}{n'})+\frac{1}{2^{p-i}}|\frac{m}{n}+\frac{m'}{n'}|
\end{align*}
To show that this holds, we need to prove the following (see inequality~\eqref{eqn:add}):
\begin{align*}
(\frac{m}{n}+\frac{m'}{n'})+\frac{1}{2^{p}}(|\frac{m}{n}|+|\frac{m'}{n'}|)\leq (\frac{m}{n}+\frac{m'}{n'})+\frac{1}{2^{p-i}}|\frac{m}{n}+\frac{m'}{n'}|\\
(\frac{m}{n}+\frac{m'}{n'})-\frac{1}{2^{p-i}}|\frac{m}{n}+\frac{m'}{n'}|\leq (\frac{m}{n}+\frac{m'}{n'})-\frac{1}{2^{p}}(|\frac{m}{n}|+|\frac{m'}{n'}|)
\end{align*}
For both inequalities, it boils down to proving the following:
\begin{align}
\frac{1}{2^i}(|\frac{m}{n}|+|\frac{m'}{n'}|)\leq |\frac{m}{n}+\frac{m'}{n'}|\label{eqn:add_diff_sgn}
\end{align}

To prove inequality~\eqref{eqn:add_diff_sgn}, we consider the following two cases:
\begin{enumerate}[1.]
\item
Suppose $\frac{m}{n}>0$ and $\frac{m'}{n'}<0$. We can rewrite inequality~\eqref{eqn:add_diff_sgn} as follows:
\begin{align*}
|\frac{m}{n}+\frac{m'}{n'}| \geq \frac{1}{2^i}(\frac{m}{n}-\frac{m'}{n'}) \Leftrightarrow & (\frac{m}{n}+\frac{m'}{n'})\geq \frac{1}{2^i}(\frac{m}{n}-\frac{m'}{n'}) \vee \\
&(\frac{m}{n}+\frac{m'}{n'})\leq \frac{1}{2^i}(\frac{m'}{n'}-\frac{m}{n}) 
\end{align*} 
In other words, we should show $\frac{\frac{m}{n}}{-\frac{m'}{n'}}\geq \frac{2^i+1}{2^i-1}$ or $\frac{-\frac{m'}{n'}}{\frac{m}{n}}\geq \frac{2^i+1}{2^i-1}$. Depending on the values of $|\frac{m}{n}|$ and $|\frac{m'}{n'}|$, both cases follow from $\frac{\max(|\frac{m}{n}|,|\frac{m'}{n'}|)}{\min(|\frac{m}{n}|,|\frac{m'}{n'}|)}\geq \frac{2^i+1}{2^i-1}$. Equivalently, we should have $i\geq \log_2(\frac{1+\frac{\min(|\frac{m}{n}|,|\frac{m'}{n'}|)}{\max(|\frac{m}{n}|,|\frac{m'}{n'}|)}}{1-\frac{\min(|\frac{m}{n}|,|\frac{m'}{n'}|)}{\max(|\frac{m}{n}|,|\frac{m'}{n'}|)}})$.
\item
Suppose $\frac{m}{n}<0$ and $\frac{m'}{n'}>0$. We can rewrite inequality~\eqref{eqn:add_diff_sgn} as follows:
\begin{align*}
|\frac{m}{n}+\frac{m'}{n'}|\geq \frac{1}{2^i}(\frac{m'}{n'}-\frac{m}{n}) \Leftrightarrow & (\frac{m}{n}+\frac{m'}{n'})\geq \frac{1}{2^i}(\frac{m'}{n'}-\frac{m}{n}) \vee \\
& (\frac{m}{n}+\frac{m'}{n'}) \leq \frac{1}{2^i}(\frac{m}{n}-\frac{m'}{n'})
\end{align*}
In other word, we need to prove $\frac{\frac{m'}{n'}}{-\frac{m}{n}}\geq \frac{2^i+1}{2^i-1}$ or $\frac{-\frac{m}{n}}{\frac{m'}{n'}}\geq \frac{2^i+1}{2^i-1}$. Depending on the values of $|\frac{m}{n}|$ and $|\frac{m'}{n'}|$, both cases follow from $\frac{\max(|\frac{m}{n}|,|\frac{m'}{n'}|)}{\min(|\frac{m}{n}|,|\frac{m'}{n'}|)}\geq \frac{2^i+1}{2^i-1}$. Equivalently, we should have $i\geq \log_2(\frac{1+\frac{\min(|\frac{m}{n}|,|\frac{m'}{n'}|)}{\max(|\frac{m}{n}|,|\frac{m'}{n'}|)}}{1-\frac{\min(|\frac{m}{n}|,|\frac{m'}{n'}|)}{\max(|\frac{m}{n}|,|\frac{m'}{n'}|)}})$.
\end{enumerate}
\end{enumerate}

\end{proof}

\begin{algorithm}
\caption{Addition}\label{alg:addition}
\begin{algorithmic}[1]
\Require $\expr$ has the shape $x+y$
\Procedure{Compute}{$\expr,p$}

\State $dp \gets p$
\Repeat \label{algline:add_go}
\State $\frac{m}{n} \gets \Call{Compute}{x,dp}$ \label{algline:addition}
\State $\frac{m'}{n'} \gets \Call{Compute}{y,dp}$
\If{$\frac{mm'}{nn'}>0$} \Comment{Theorem~\ref{thm:add}.\ref{thm:add_same_sign}}
	\State \Return $\frac{mn'+m'n}{nn'}$
\Else \Comment{Theorem~\ref{thm:add}.\ref{thm:add_diff_sign}}
	\State $i\gets \lceil \log_2(\frac{1+\frac{\min(|\frac{m}{n}|,|\frac{m'}{n'}|)}{\max(|\frac{m}{n}|,|\frac{m'}{n'}|)}}{1-\frac{\min(|\frac{m}{n}|,|\frac{m'}{n'}|)}{\max(|\frac{m}{n}|,|\frac{m'}{n'}|)}}) \rceil$
	\If{$dp-i\geq p$}
		\State \Return $\frac{mn'+m'n}{nn'}$
	\Else
		\State $dp\gets dp+1$ \label{algline:add_inc}
	\EndIf
\EndIf
\Until{$\true$} 
\EndProcedure
\end{algorithmic}
\end{algorithm}

Algorithm~\ref{alg:addition} applies Theorem~\ref{thm:add} to approximate $x+y$ with precision $p$. If approximations $\frac{m}{n}$ and $\frac{m'}{n'}$ have the same sign, we do not lose precision by calculating $x+y$. On the other hand, if $\frac{m}{n}$ and $\frac{m'}{n'}$ have different signs, the amount of precision that is lost depends on the magnitude of $\frac{\min (|\frac{m}{n}|,|\frac{m'}{n'}|)}{\max(|\frac{m}{n}|,|\frac{m'}{n'}|)}$. This indicates that if $\frac{mm'}{nn'}<0$ and $|\frac{m}{n}|\approx |\frac{m'}{n'}|$, a significant amount of precision can be lost in $x+y$. Thus, if the guaranteed precision for $x+y$ (i.e., $p-i$) is not sufficient, $x$ and $y$ must be recomputed with higher precisions (see Line~\ref{algline:add_go},\ref{algline:add_inc} in Algorithm~\ref{alg:addition}).  

To confirm this observation, we apply the perturbation analysis of Section~\ref{sec:condition} on $f(x,y)=x+y$:
\begin{align*}
|\frac{xf_x(x,y)}{f(x,y)}|+|\frac{yf_y(x,y)}{f(x,y)}|=|\frac{x}{x+y}|+|\frac{y}{x+y}|
\end{align*}

The quantity $|\frac{x}{x+y}|+|\frac{y}{x+y}|$ is $1$ when $xy>0$. Thus, we can estimate the loss of precision in $x+y$ independently of the arguments when $xy>0$. 

However, the quantity $|\frac{x}{x+y}|+|\frac{y}{x+y}|$ can be arbitrarily large when $xy<0$ and $|x+y|\approx 0$. This confirms that a significant amount of precision can be lost in $x+y$. 

Perturbation analysis shows that in general, we cannot estimate the amount of precision that is lost in $x+y$ independently of the arguments. Hence, if approximation $\frac{m}{n}$ and $\frac{m'}{n'}$ have different signs, recomputing $x$ and $y$  might be essential to obtain the desired precision for $x+y$. Loss of precision in $x+y$ is sometimes referred to as loss of significance \cite{KC02} or catastrophic cancellation~\cite{FT86}.

It should be noted that Theorem~\ref{thm:add} does not imply that $x+y$ is always fundamentally problematic when $xy<0$ and $|x+y|\approx 0$. In certain cases, the calculation can be adjusted in such a way that loss of significance can be avoided and the expression can be calculated in one pass.

Suppose we want to calculate $\sqrt{x+1}-\sqrt{x}$ for a relatively large $x$. Since $\sqrt{x+1}\approx\sqrt{x}$, we will lose a significant amount of precision if we directly calculate $\sqrt{x+1}-\sqrt{x}$. However, we can change the calculation algorithm by rewriting the expression as follows:
\begin{align*}
\sqrt{x+1}-\sqrt{x}=(\sqrt{x+1}-\sqrt{x})\times \frac{\sqrt{x+1}+\sqrt{x}}{\sqrt{x+1}+\sqrt{x}}=\frac{1}{\sqrt{x+1}+\sqrt{x}}
\end{align*}
In the new expression, all the operations can be approximated with a desired precision in one pass (see Section~\ref{subsec:sqrt} on calculating square root). Hence, we can approximate the new expression without recomputing the sub-expressions with higher precisions. Applying the perturbation analysis of Section~\ref{sec:condition} also shows that $f(x)=\sqrt{x+1}-\sqrt{x}$ is not fundamentally problematic:
\begin{align*}
|\frac{xf'(x)}{f(x)}|=|\frac{x(\frac{1}{2\sqrt{x+1}}-\frac{1}{2\sqrt{x}})}{\sqrt{x+1}-\sqrt{x}}|=\frac{1}{2}\sqrt{\frac{x}{x+1}}<\frac{1}{2}
\end{align*}

Using perturbation analysis, we can identify instances of $x+y$ where adjustments in the algorithm can avoid loss of significance. However, to our knowledge a general scheme for making such adjustments does not exist.

\subsection{Square Root}
\label{subsec:sqrt}
In this section, we calculate $\sqrt{x}$ by approximating the root of $f(y)=y^2-x$ using the Newton-Raphson method \cite{KC02}.
The Newton-Raphson method starts with an initial approximation $y_0$ for $\sqrt{x}$ and iteratively generates a sequence of approximations. Assuming that the precise value of $x$ is available, the sequence of approximations is generated by:
\begin{align*}
y_{n+1}=\frac{y_n^2+x}{2y_n}
\end{align*}
In what follows, we prove a theorem for approximating $\sqrt{x}$ by the Newton-Raphson method when an approximation $(m,n,p)$ of $x$ is available. 
 
\begin{mytheorem}\label{thm:sqrt}
Let $x$ be a real number represented by $(m,n,p)$ such that:
\begin{align*}
\frac{m}{n}=0.b_1\ldots b_k\times 2^a
\end{align*}
where $b_i\in\{0,1\}$ for $1\leq i\leq k$ and $b_1=1, a\in \integer$. Then $\sqrt{x}$ can be represented by $(m',n',p-4N)$ where $N=\lceil \sqrtexp\rceil+1$ and $\frac{m'}{n'}$ is the $N$-th term of the following sequence:
\begin{align*}
y_{n+1}=\frac{y_n^2+\frac{m}{n}}{2y_n},~y_0=2^{\lceil\frac{a}{2} \rceil +1}
\end{align*}

\end{mytheorem}

\begin{proof}
From Definition~\ref{def:realrep} we write:
\begin{align*}
\frac{m}{n}(1-\frac{1}{2^{p}})<x<\frac{m}{n}(1+\frac{1}{2^{p}})
\end{align*}
To prove the theorem, it suffices to show:
\begin{align}
|\sqrt{x}-y_N|<\frac{1}{2^{p-4N}}|y_N|\label{eqn:sqrt}
\end{align}
We rewrite the left hand side of inequality~\eqref{eqn:sqrt}:
\begin{align}
|\sqrt{x}-y_N|=|\sqrt{x}-z_N+z_N-y_N|\leq |\sqrt{x}-z_N|+|z_N-y_N|\label{eqn:sqrt_rewrite}
\end{align}
In inequality~\eqref{eqn:sqrt_rewrite}, $z_N$ is the $N$-th term of the following sequence:
\begin{align*}
z_{n+1}=\frac{z_n^2+x}{2z_n},~z_0=2^{\lceil\frac{a}{2} \rceil+1}
\end{align*}
To prove inequality~\eqref{eqn:sqrt}, it suffices to show that the following inequalities hold:
\begin{align}
|\sqrt{x}-z_N|<\frac{1}{2^{p-4N+1}}|y_N|\label{eqn:sqrt_part_1}\\
|z_N-y_N|<\frac{1}{2^{p-4N+1}}|y_N|\label{eqn:sqrt_part_2}
\end{align}
\textbf{Proof for inequality~\eqref{eqn:sqrt_part_1}:} 
First, we show that $y_n>\sqrt{\frac{x}{2}}$ for all $n\in \nat$:
\begin{align*}
y_n-\sqrt{\frac{x}{2}}=\frac{y_{n-1}^2+\frac{m}{n}}{2y_{n-1}}-\sqrt{\frac{x}{2}}&>\frac{y_{n-1}^2+\frac{m}{n}}{2y_{n-1}}-\sqrt{\frac{m}{2n}(1+\frac{1}{2^p})}\geq \frac{y_{n-1}^2+\frac{m}{n}}{2y_{n-1}}-\sqrt{\frac{m}{n}}\\
&=\frac{y_{n-1}^2+\frac{m}{n}-2y_{n-1}\sqrt{\frac{m}{n}}}{2y_{n-1}}=\frac{(y_{n-1}-\sqrt{\frac{m}{n}})^2}{2y_{n-1}}\geq 0
\end{align*} 
Observe that as $y_n>\sqrt{\frac{x}{2}}$, inequality~\eqref{eqn:sqrt_part_1} is valid, if the following inequality holds:
\begin{align}
|\sqrt{x}-z_N|<\frac{1}{2^{p-4N+\frac{3}{2}}}\sqrt{x}\label{eqn:sqrt_part_1_rewrite}
\end{align}
To prove inequality~\eqref{eqn:sqrt_part_1_rewrite}, we find $N$ such that:
\begin{align}
z_N-\sqrt{x}=\frac{1}{2^{p+2}}\sqrt{x}\label{eqn:seq_prec}
\end{align} 
We calculate the quantity  $\frac{z_N+\sqrt{x}}{z_N-\sqrt{x}}$:
\begin{align}
\frac{z_N+\sqrt{x}}{z_N-\sqrt{x}}=&\frac{\frac{z_{N-1}^2+x}{2z_{N-1}}+\sqrt{x}}{\frac{z_{N-1}^2+x}{2z_{N-1}}-\sqrt{x}}=\frac{z_{N-1}^2+x+2z_{N-1}\sqrt{x}}{z_{N-1}^2+x-2z_{N-1}\sqrt{x}}\nonumber\\
=&(\frac{z_{N-1}+\sqrt{x}}{z_{N-1}-\sqrt{x}})^2=\ldots=(\frac{z_{0}+\sqrt{x}}{z_{0}-\sqrt{x}})^{2^{N}}\label{eqn:sqrt_recur}
\end{align}
Suppose equality~\eqref{eqn:seq_prec} holds for $N$. We can rewrite the right hand side of equality~\eqref{eqn:sqrt_recur} as follows:
\begin{align}
(\frac{z_{0}+\sqrt{x}}{z_{0}-\sqrt{x}})^{2^{N}}=\frac{z_N+\sqrt{x}}{z_N-\sqrt{x}}=\frac{(2+\frac{1}{2^{p+2}})\sqrt{x}}{(\frac{1}{2^{p+2}})\sqrt{x}}=2^{p+3}+1\label{eqn:nr_iteration}
\end{align}
The index $N$ that guarantees the required precision of equality~\eqref{eqn:seq_prec} can be calculated from equality~\eqref{eqn:nr_iteration}:
\begin{align}
&2^{N}\log_2 (\frac{z_{0}+\sqrt{x}}{z_{0}-\sqrt{x}})=\log_2 (2^{p+3}+1)\nonumber\\
&N=\log_2 (\log_2 (2^{p+3}+1))- \log_2 (\log_2 (\frac{z_{0}+\sqrt{x}}{z_{0}-\sqrt{x}}))\label{eqn:nr_iteration_rewrite}
\end{align}
To guarantee that $N$ is well-defined, we show that $z_0>x$. To prove the inequality, we use the assumptions $\frac{m}{n}=0.b_1\ldots b_k\times 2^a, b_i\in \{0,1\} \text{ for }i=1,\ldots , k$ and $b_0=1, z_0=2^{\lceil \frac{a}{2}\rceil +1}$:
\begin{align*}
\sqrt{x}<\sqrt{\frac{m}{n}(1+\frac{1}{2^p})}\leq\sqrt{0.b_1\ldots b_k}\times 2^{\frac{a}{2}}\times \sqrt{2}<2^{\frac{a}{2}}\times \sqrt{2}<z_0
\end{align*}
To obtain an estimation for $N$ from equality~\eqref{eqn:nr_iteration_rewrite}, we calculate an upper bound for $\frac{z_0}{\sqrt{x}}$:
\begin{align*}
\frac{z_0}{\sqrt{x}}<\frac{2^{\lceil \frac{a}{2} \rceil+1}}{\sqrt{\frac{m}{n}(1-\frac{1}{2^{p}})}}<\frac{2^{\frac{a+1}{2} +1}}{\sqrt{0.b_1\ldots b_k}\times 2^{\frac{a}{2}}\times \sqrt{\frac{1}{2}}}\leq \frac{2\sqrt{2}}{\sqrt[4]{2}\times\sqrt{\frac{1}{2}}}=\frac{4}{\sqrt[4]{2}}
\end{align*}
In the worst case, the initial approximation $z_0$ differs from $\sqrt{x}$ by a factor $\frac{4}{\sqrt[4]{2}}$. We use this estimation in equality~\eqref{eqn:nr_iteration_rewrite} to calculate the number of iterations for the Newton-Raphson method:
\begin{align}
N=\log_2 (\log_2 (2^{p+3}+1))-& \log_2 (\log_2 (\frac{\frac{4}{\sqrt[4]{2}}+1}{\frac{4}{\sqrt[4]{2}}-1}))\nonumber\\
<& \lceil\sqrtexp \rceil+1\label{eqn:n_upper}
\end{align}
\textbf{Proof for inequality~\eqref{eqn:sqrt_part_2}:} To prove inequality~\eqref{eqn:sqrt_part_2}, we consider the calculations in $z_N=\frac{z_{N-1}^2+x}{2z_{N-1}}$ and estimate the amount of error that we commit in the approximation $y_N=\frac{y_{N-1}^2+\frac{m}{n}}{2y_{N-1}}$. 

Let $P(k)$ denote the amount of precision that we lose when we approximate $z_k$ by $y_k$. Thus, we lose $P(N-1)$ units of precision if we approximate $z_{N-1}$ by $y_{N-1}$:
\begin{align*}
|z_{N-1}-y_{N-1}|<\frac{1}{2^{p-P(N-1)}}|y_{N-1}|
\end{align*}
The precision is reduced by $2$ units when $z_{N-1}^2$ is approximated by $y_{N-1}^2$ (see Theorem~\ref{thm:mult}):
\begin{align*}
|z_{N-1}^2-y_{N-1}^2|<\frac{1}{2^{p-P(N-1)-2}}|y_{N-1}^2|
\end{align*}
We approximate $z_{N-1}^2+x$ by $y_{N-1}^2+\frac{m}{n}$. We do not lose precision in this approximation (see Theorem~\ref{thm:add}.\ref{thm:add_same_sign}):
\begin{align}
|(z_{N-1}^2+x)-(y_{N-1}^2+\frac{m}{n})|<\frac{1}{2^{p-P(N-1)-2}}|y_{N-1}^2+\frac{m}{n}|\label{eqn:sqrt_ineq1}
\end{align}
Given the approximation $y_{N-1}$ of $z_{N-1}$, one unit of precision is lost in the approximation of $\frac{1}{z_{N-1}}$ (see Theorem~\ref{thm:inverse}):
\begin{align}
|\frac{1}{z_{N-1}}-\frac{1}{y_{N-1}}|&<\frac{1}{2^{p-P(N-1)-1}}|\frac{1}{y_{N-1}}|\nonumber \\
|\frac{1}{2z_{N-1}}-\frac{1}{2y_{N-1}}|&<\frac{1}{2^{p-P(N-1)-1}}|\frac{1}{2y_{N-1}}|\label{eqn:sqrt_ineq2}
\end{align}
Finally, we approximate $\frac{z_{N-1}^2+x}{2z_{N-1}}$ based on the approximations described in inequality~\eqref{eqn:sqrt_ineq1} and \eqref{eqn:sqrt_ineq2} (see Theorem~\ref{thm:mult}):
\begin{align}
|z_N-y_N|=|\frac{z_{N-1}^2+x}{2z_{N-1}}-\frac{y_{N-1}^2+\frac{m}{n}}{2y_{N-1}}|<\frac{1}{2^{p-P(N-1)-4}}|\frac{y_{N-1}^2+\frac{m}{n}}{2y_{N-1}}|\label{eqn:sqrt_ineq3}
\end{align}
From inequality~\eqref{eqn:sqrt_ineq3} we obtain the following recursive formula:
\begin{align*}
P(N)=P(N-1)+4
\end{align*}
Since $y_0=z_0=2^{\lceil\frac{a}{2} \rceil +1}$, we lose $P(N)=P(0)+4N=4N$ units of precision in our approximation of $z_N$. We apply the number of iterations calculated in inequality~\eqref{eqn:n_upper} and obtain:
\begin{align*}
P(N)=4N<4\lceil\sqrtexp \rceil+4
\end{align*}

\end{proof}

\begin{algorithm}
\caption{Square Root}\label{alg:square_root}
\begin{algorithmic}[1]
\Require $\expr$ has the shape $\sqrt{x}$
\Procedure{Compute}{$\expr,p$}

\State Choose $p_x$ such that $p_x\geq p +4\lceil\log_2\big(\log_2(2^{p_x+3}+1)\big) \rceil+4$
\State $N\gets \lceil\log_2\big(\log_2(2^{p_x+3}+1)\big) \rceil+1$
\State $\frac{m}{n}\gets \Call{Compute}{x,p_x}$
\If{$\frac{m}{n}<0$}
	\State ``Undefined operation''
\Else
	\LineComment{$\frac{m}{n}$ can be represented as $0.b_1\ldots b_k\times 2^a$}
 	\LineComment{$b_i\in\{0,1\}$ for $1\leq i\leq k, b_1=1$ and $a\in\integer$}
	\State $a\gets \lfloor\log_2(\frac{m}{n}) \rfloor +1$
	\State $y_0\gets 2^{\lceil\frac{a}{2} \rceil +1}$
	\For{$i=1$ to $N$} 
		\State $y_{i}\gets \frac{y_{i-1}^2+\frac{m}{n}}{2y_{i-1}}$ 
	\EndFor
	\State $\frac{m'}{n'}\gets y_N$
	\State \Return $\frac{m'}{n'}$
\EndIf
\EndProcedure
\end{algorithmic}
\end{algorithm}

Algorithm~\ref{alg:square_root} applies Theorem~\ref{thm:sqrt} to approximate $\sqrt{x}$ with precision $p$. As indicated in Theorem~\ref{thm:sqrt}, loss of precision in the square root can be estimated independently of the argument $x$ and hence Algorithm~\ref{alg:square_root} approximates $\sqrt{x}$ in one pass. We apply perturbation analysis on $f(x)=\sqrt{x}$ to show this:
\begin{align*}
|\frac{xf'(x)}{f(x)}|=|\frac{(x)(\frac{1}{2\sqrt{x}})}{\sqrt{x}}|=\frac{1}{2}
\end{align*}
The quantity $|\frac{xf'(x)}{f(x)}|$ is a small constant. Thus, in our top down approach, we can approximate $\sqrt{x}$ without iterative computations. 

%% file: transcendental_func_riemann.tex
\section{Approximating Transcendental Functions by Riemann Sums}\label{sec:riemann}
In this section we introduce approximations for $e^x$ (Section~\ref{subsec:exp}), $\ln(x)$ (Section~\ref{subsec:log}), and $\arctan(x)$ (Section~\ref{subsec:arctan}). We use Riemann sums to approximate these functions. 

The \textproc{Compute}$(expr,p)$ function introduced in Section~\ref{sec:algebraic_operations} will be extended to allow the approximation of $e^x, \ln(x)$ and $\arctan(x)$ with a given precision $p$. Perturbation analysis will also be used to identify computational problems in which iterative computations are unavoidable. 

\subsection{Exponential Function}\label{subsec:exp}
To approximate the exponential function by Riemann sums, we first provide a simple approximation for $e^x$ where we assume that $x$ is precise. Then we extend this calculation to approximate $e^x$ where $x$ is represented by $(m,n,p)$.

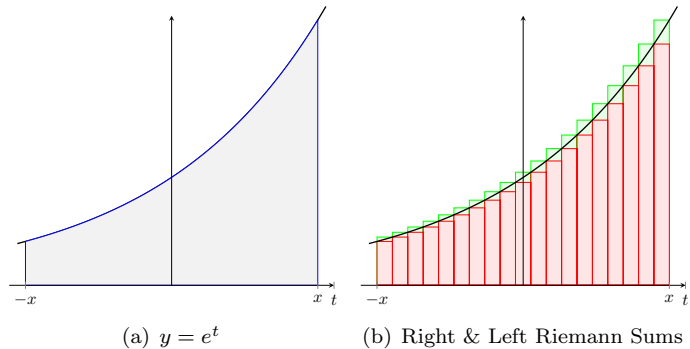
\begin{figure}[t]
\centering
\subfigure[$y=e^t$]{\label{fig:exp}
\begin{tikzpicture}[scale=0.63]
\begin{axis}[
	ytick=\empty,
	xtick=\empty,
    xmax=1,ymax=2.5,ymin=0,xmin=-1,
    enlargelimits=true,
    axis lines=middle,
    clip=false,
    every extra x tick/.style={
        grid=none, 
        tick0/.initial=red,
        tick1/.initial=green,
        tick2/.initial=orange,
        xticklabel style={
            anchor=north,
        },
    },
    extra x ticks ={-0.9,0.9,1},
    extra x tick labels={$-x$,$x$,$t$},   
    domain=-3:3,
    axis on top
    ]
\addplot[smooth, thick,domain=-0.95:0.95,samples=40]{e^x};
\addplot+[mark=none,
        domain=-0.9:0.9,
        samples=100,
        draw=blue,
        fill=gray!10
        ]{e^x} \closedcycle; 
\end{axis}
\end{tikzpicture}}
\subfigure[Right \& Left Riemann Sums]{\label{fig:exp_riemann}
\begin{tikzpicture}[scale=0.63]
\begin{axis}[
	ytick=\empty,
	xtick=\empty,
    xmax=1,ymax=2.5,ymin=0,xmin=-1,
    enlargelimits=true,
    axis lines=middle,
    clip=false,
    every extra x tick/.style={
        grid=none, 
        tick0/.initial=red,
        tick1/.initial=green,
        tick2/.initial=orange,
        xticklabel style={
            anchor=north,
        },
    },
    extra x ticks ={-0.9,0.9,1},
    extra x tick labels={$-x$,$x$,$t$},   
    domain=-3:3,
    axis on top
    ]
\addplot [draw=green, fill=green!10, ybar interval, samples=20, domain=0.9:-0.9]
    {e^x}\closedcycle;
\addplot [draw=red, fill=red!10, ybar interval, samples=20, domain=-0.9:0.9]
    {e^x}\closedcycle;
\addplot[smooth, thick,domain=-0.95:0.95,samples=40]{e^x};
\end{axis}
\end{tikzpicture}}
\caption{Approximating $e^x$}
\end{figure}

Suppose $x>0$. To calculate $e^x$ we consider the curve $y=e^t$ and calculate the area enclosed by this curve and the t-axis between $t=-x$ and $t=x$ as follows (see Fig.~\ref{fig:exp}):
\begin{align*}
\int_{-x}^{x} e^t dt=e^x-e^{-x}
\end{align*}
We use Riemann sums to approximate this area; Fig.~\ref{fig:exp_riemann} shows two approximations from above and below using rectangles. Thus, we get the following inequalities for $N$ rectangles:
\begin{align}
\sum_{i=0}^{N-1}\frac{2x}{N}e^{-x+\frac{2x}{N}i}\leq e^x-e^{-x}\leq \sum_{i=1}^{N}\frac{2x}{N}e^{-x+\frac{2x}{N}i}\label{eqn:exp_sum}
\end{align}
We rewrite inequality~\eqref{eqn:exp_sum} as follows:
\begin{align}
\sum_{i=0}^{N-1}(\frac{2x}{N})e^{\frac{2x}{N}i}&\leq e^{2x}-1\leq \sum_{i=1}^{N}(\frac{2x}{N})e^{\frac{2x}{N}i}\nonumber\\
(\frac{2x}{N})(\frac{e^{2x}-1}{e^{\frac{2x}{N}}-1})&\leq e^{2x}-1\leq (\frac{2x}{N})(\frac{e^{\frac{2x}{N}}(e^{2x}-1)}{e^{\frac{2x}{N}}-1})\nonumber\\
\frac{2x}{N}&\leq e^{\frac{2x}{N}}-1 \leq \frac{2x}{N}e^{\frac{2x}{N}}\label{eqn:exp_sum_rewrite}
\end{align}
We assume that $N>2x$ and calculate an upper bound and a lower bound for $e^x$ from inequality~\eqref{eqn:exp_sum_rewrite}:
\begin{align}
(1+\frac{2x}{N})^{\frac{N}{2}} \leq e^x\leq (\frac{N}{N-2x})^{\frac{N}{2}}\label{eqn:exp_approx}
\end{align}

We can estimate the precision of the approximations calculated in inequality~\eqref{eqn:exp_approx}. For example, we can approximate $e^x$ by $(\frac{N}{N-2x})^{\frac{N}{2}}$ and the absolute error of this approximation can be calculated as follows:
\begin{align}
|(\frac{N}{N-2x})^{\frac{N}{2}}-e^x|&\leq |(\frac{N}{N-2x})^{\frac{N}{2}}-(1+\frac{2x}{N})^{\frac{N}{2}}|\nonumber\\
&=((\frac{N}{N-2x})-(\frac{N+2x}{N}))\sum_{i=0}^{\frac{N}{2}-1}(\frac{N}{N-2x})^i(\frac{N+2x}{N})^{\frac{N}{2}-i-1}\nonumber\\
&=\frac{4x^2}{N(N-2x)}\sum_{i=0}^{\frac{N}{2}-1}(\frac{N}{N-2x})^i(\frac{N+2x}{N})^{\frac{N}{2}-i-1}\nonumber\\
&\leq \frac{4x^2}{N(N-2x)}\sum_{i=0}^{\frac{N}{2}-1}(\frac{N}{N-2x})^{\frac{N}{2}-1}=\frac{2x^2}{N}(\frac{N}{N-2x})^{\frac{N}{2}}\label{eqn:exp_riemann_prec}
\end{align}
For the last inequality we apply $\frac{N}{N-2x} \geq \frac{N+2x}{N}$.

In the discussion above, we have treated $x$ as a precise value. In the following theorem, we extend this calculation and describe an approximation of $e^x$ that relies on a representation $(m,n,p)$ of $x$. To simplify our approximations, we first assume that $|\frac{m}{n}|<1$. Afterwards, we extend our approximations to an arbitrary $(m,n,p)$.
\begin{mytheorem}\label{thm:base_exp}
Let $x$ be a real number represented by $(m,n,p)$ and $|\frac{m}{n}|<1$. Suppose $N$ is a natural number such that $N>2^{\frac{p+11}{3}}$. The value of $e^x$ can be approximated as follows:
\begin{enumerate}[i.]
\item \label{thm:base_exp_1}
If $0<\frac{m}{n}<1$ then $e^x$ can be represented by $(m',n',p-2\lceil\log_2(\frac{N}{2}) \rceil-3)$ where $\frac{m'}{n'}=(\frac{N}{N-\frac{2m}{n}})^{\frac{N}{2}}$.
\item \label{thm:base_exp_2}
If $-1<\frac{m}{n}<0$ then $e^x$ can be represented by $(m',n',p-2\lceil\log_2(\frac{N}{2}) \rceil-4)$ where $\frac{m'}{n'}=\frac{1}{(\frac{N}{N+\frac{2m}{n}})^{\frac{N}{2}}}$.
\end{enumerate}
\end{mytheorem}

\begin{proof}
\textit{\\}
\begin{enumerate}[i.]
\item
Suppose $0<\frac{m}{n}<1$. Since $x$ is represented by $(m,n,p)$, we have:
\begin{align}
0<\frac{m}{n}(1-\frac{1}{2^p})<x<\frac{m}{n}(1+\frac{1}{2^p})\leq \frac{2m}{n}<2\label{eqn:exp_bound}
\end{align}
To prove the theorem, it suffices to show that:
\begin{align}
|e^x-(\frac{N}{N-\frac{2m}{n}})^{\frac{N}{2}}|<\frac{1}{2^{p-2\lceil \log_2(\frac{N}{2})\rceil-3}}|(\frac{N}{N-\frac{2m}{n}})^{\frac{N}{2}}|\label{eqn:exp}
\end{align}
We rewrite the left hand side of inequality~\eqref{eqn:exp} as follows:
\begin{align*}
|e^x-(\frac{N}{N-\frac{2m}{n}})^{\frac{N}{2}}|&=|e^x-(\frac{N}{N-2x})^{\frac{N}{2}}+(\frac{N}{N-2x})^{\frac{N}{2}}-(\frac{N}{N-\frac{2m}{n}})^{\frac{N}{2}}|\\
&\leq |e^x-(\frac{N}{N-2x})^{\frac{N}{2}}|+|(\frac{N}{N-2x})^{\frac{N}{2}}-(\frac{N}{N-\frac{2m}{n}})^{\frac{N}{2}}|
\end{align*}
To prove inequality~\eqref{eqn:exp}, it suffices to show that the following inequalities hold:
\begin{align}
&|e^x-(\frac{N}{N-2x})^{\frac{N}{2}}|<\frac{1}{2^{p-2\lceil \log_2(\frac{N}{2})\rceil-2}}|(\frac{N}{N-\frac{2m}{n}})^{\frac{N}{2}}|\label{eqn:exp_part_1} \\
&|(\frac{N}{N-2x})^{\frac{N}{2}}-(\frac{N}{N-\frac{2m}{n}})^{\frac{N}{2}}|<\frac{1}{2^{p-2\lceil\log_2(\frac{N}{2}) \rceil-2}}|(\frac{N}{N-\frac{2m}{n}})^{\frac{N}{2}}|\label{eqn:exp_part_2}
\end{align}
\textbf{Proof for inequality~\eqref{eqn:exp_part_1}:} %To prove the inequality it suffices to show that:
%\begin{align}
%|e^x-(\frac{N}{N-2x})^{\frac{N}{2}}|<\frac{1}{2^{p-2}}|(\frac{N}{N-\frac{2m}{n}})^{\frac{N}{2}}|\label{eqn:exp_part_1_rewrite}
%\end{align}
From inequality~\eqref{eqn:exp_riemann_prec}, we obtain an upper-bound for the left hand side of inequality~\eqref{eqn:exp_part_1}:
\begin{align}
|e^x-(\frac{N}{N-2x})^{\frac{N}{2}}|\leq \frac{2x^2}{N}(\frac{N}{N-2x})^{\frac{N}{2}}\label{eqn:exp_part1_upper}
\end{align} 
To calculate an upper bound for the right hand side of inequality~\eqref{eqn:exp_part1_upper}, we consider the function $f(x)=\frac{2x^2}{N}(\frac{N}{N-2x})^{\frac{N}{2}}$ and calculate its derivative:
\begin{align*}
f'(x)=\frac{4x}{N}(\frac{N}{N-2x})^{\frac{N}{2}}+x^2(\frac{N}{N-2x})^{\frac{N}{2}-1}(\frac{2N}{(N-2x)^2})
\end{align*}
From inequality~\eqref{eqn:exp_bound} we obtain $x\in(0,2)$. We choose:
\begin{align}
N> 4>2x\label{eqn:exp_bound1}
\end{align}
to ensure that $f(x)$ is increasing for $x\in (0,2)$, i.e., $f'(x)>0$. We rewrite inequality~\eqref{eqn:exp_part1_upper} as follows:
\begin{align}
|e^x-(\frac{N}{N-2x})^{\frac{N}{2}}|\leq f(x) \leq f(2)=(\frac{8}{N})(\frac{N}{N-4})^{\frac{N}{2}}\label{eqn:exp_part_1_upper1}
\end{align}
We calculate a lower bound for the right hand side of inequality~\eqref{eqn:exp_part_1} as follows:
\begin{align}
\frac{1}{2^{p-2\lceil\log_2(\frac{N}{2}) \rceil-2}}|(\frac{N}{N-\frac{2m}{n}})^{\frac{N}{2}}|>\frac{1}{2^{p-2\log_2(\frac{N}{2})-2}}\label{eqn:exp_part_1_lower}
\end{align}

Thus, to prove inequality~\eqref{eqn:exp_part_1} it suffices to show that the following inequality holds (see inequality~\eqref{eqn:exp_part_1_upper1}, \eqref{eqn:exp_part_1_lower}):
\begin{align*}
(\frac{8}{N})(\frac{N}{N-4})^{\frac{N}{2}}<\frac{1}{2^{p-2\log_2(\frac{N}{2})-2}}
\end{align*}
This is equivalent to the following:
\begin{align}
3-\log_2(N)+\frac{N}{2}\big(\log_2(1+\frac{4}{N-4})\big)<-p+2\log_2(\frac{N}{2})+2\label{eqn:exp_n}
\end{align}
We choose $N>8$ and apply Proposition~\ref{prop:ln_bounds} (see \ref{app:prop_lem}) to obtain an upper bound for $\log_2(1+\frac{4}{N-4})$:
\begin{align}
\log_2(1+\frac{4}{N-4})<\frac{4}{(N-4)\ln(2)}<\frac{8}{N-4}\label{eqn:exp_n_upper}
\end{align}
Based on inequality~\eqref{eqn:exp_n},\eqref{eqn:exp_n_upper}, it is sufficient to find an $N>8$ satisfying: 
\begin{align}
3-\log_2(N)+(\frac{N}{2})(\frac{8}{N-4})<-p+2\log_2(\frac{N}{2})+2\label{eqn:exp_part_1_rewrite_1}
\end{align}
Inequality~\eqref{eqn:exp_part_1_rewrite_1} is equivalent to the following:
\begin{align*}
-\log_2(\frac{N^3}{4})+\frac{16}{N-4}<-p-5
\end{align*}
From $N>8$, we conclude $\frac{16}{N-4}<4$. Thus, we choose $N$ such that $N>\max(2^{\frac{p+11}{3}},8)=2^{\frac{p+11}{3}}$.

\textbf{Proof for inequality~\eqref{eqn:exp_part_2}:} To prove the inequality, we estimate the amount of precision that is lost when we approximate $(\frac{N}{N-2x})^{\frac{N}{2}}$ by $(\frac{N}{N-\frac{2m}{n}})^{\frac{N}{2}}$.

The number $x$ is represented by $(m,n,p)$. Thus, we have:
\begin{align*}
&|x-\frac{m}{n}|<\frac{1}{2^p}|\frac{m}{n}|\\
&|2x-\frac{2m}{n}|<\frac{1}{2^p}|\frac{2m}{n}|
\end{align*}
Since $N>8$, one unit of precision is lost when we approximate $N-2x$ by $N-\frac{2m}{n}$ (see Theorem~\ref{thm:add}.\ref{thm:add_diff_sign}):
\begin{align*}
|(N-2x)-(N-\frac{2m}{n})|<\frac{1}{2^{p-1}}|N-\frac{2m}{n}|
\end{align*}
Approximating $\frac{1}{N-2x}$ by $\frac{1}{N-\frac{2m}{n}}$ reduces the precision by one unit (see Theorem~\ref{thm:inverse}):
\begin{align*}
|\frac{1}{N-2x}-\frac{1}{N-\frac{2m}{n}}|<\frac{1}{2^{p-2}}|\frac{1}{N-\frac{2m}{n}}|\\
|\frac{N}{N-2x}-\frac{N}{N-\frac{2m}{n}}|<\frac{1}{2^{p-2}}|\frac{N}{N-\frac{2m}{n}}|
\end{align*}
Finally, approximating $(\frac{N}{N-2x})^{\frac{N}{2}}$ reduces the precision by $2\lceil \log_2(\frac{N}{2}) \rceil$ units (see Lemma~\ref{lm:power} in \ref{app:prop_lem}):
\begin{align*}
|(\frac{N}{N-2x})^{\frac{N}{2}}-(\frac{N}{N-\frac{2m}{n}})^{\frac{N}{2}}|<\frac{1}{2^{p-2\lceil\log_2(\frac{N}{2}) \rceil -2}}|(\frac{N}{N-\frac{2m}{n}})^{\frac{N}{2}}|
\end{align*}
\item
Suppose $\frac{m}{n}<0$. We use the following identity to calculate $e^x$:
\begin{align*}
e^x=\frac{1}{e^{-x}}
\end{align*}
We represent $-x$ by $(-m,n,p)$ (see Theorem~\ref{thm:unary_neg}). Then, we apply the first part of the theorem and Theorem~\ref{thm:inverse} to approximate $e^{-x}$ and $\frac{1}{e^{-x}}$, respectively.
\end{enumerate}
\end{proof}

In what follows, we extend the approximations of Theorem~\ref{thm:base_exp} and calculate the exponential function for $x$ represented by $(m,n,p)$ where $|\frac{m}{n}|\geq 1$.
\begin{mytheorem}\label{thm:extend_exp}
Let $x$ be a real number represented by $(m,n,p)$ and $|\frac{m}{n}|\geq 1$. Suppose $k$ and $N$ are natural numbers such that:
\begin{align*}
|\frac{m}{2^kn}|<1~,~N>2^{\frac{p+11}{3}}
\end{align*}
The value of $e^x$ can be approximated as follows:
\begin{enumerate}[i.]
\item \label{thm:extend_exp_1}
If $\frac{m}{n}>0$ then $e^x$ can be represented by $(m',n',p-2\lceil\log_2(\frac{N}{2}) \rceil-2k-3)$ where $\frac{m'}{n'}=(\frac{N}{N-\frac{2m}{n}})^{N\cdot 2^{k-1}}$.
\item \label{thm:extend_exp_2}
If $\frac{m}{n}<0$ then $e^x$ can be represented by $(m',n',p-2\lceil\log_2(\frac{N}{2}) \rceil-2k-4)$ where $\frac{m'}{n'}=\frac{1}{(\frac{N}{N+\frac{2m}{n}})^{N\cdot 2^{k-1}}}$.
\end{enumerate}
\end{mytheorem}

\begin{proof}
Since $|\frac{m}{n}|\geq1$, we choose $k\in\nat$ such that $|\frac{m}{2^kn}|<1$. We use the following identity to calculate $e^x$:
\begin{align}
e^x=(e^{\frac{x}{2^k}})^{2^k} \label{eqn:exp_range}
\end{align} 
We approximate $\frac{x}{2^k}$ by $\frac{m}{2^kn}$. Since $2^k$ is a constant, we do not lose precision in this approximation. We apply Theorem~\ref{thm:base_cos} to approximate $e^\frac{x}{2^k}$. This approximation reduces the precision by:
\begin{itemize}
\item
$2\lceil\log_2(\frac{N}{2}) \rceil +3$ units, if $0<\frac{m}{2^kn}<1$;
\item
$2\lceil\log_2(\frac{N}{2}) \rceil +4$ units, if $-1<\frac{m}{2^kn}<0$.
\end{itemize}
Suppose $\frac{m'}{n'}$ is the approximation obtained for $e^{\frac{x}{2^k}}$ from Theorem~\ref{thm:base_exp}. We approximate $(e^{\frac{x}{2^k}})^{2^k}$ by $(\frac{m'}{n'})^{2^k}$; we lose $2k$ units of precision in this calculation (see Lemma~\ref{lm:power} in \ref{app:prop_lem}).

\end{proof}

\begin{algorithm}
\caption{Exponential Function}\label{alg:exp}
\begin{algorithmic}[1]
\Require $\expr$ has the shape $e^x$
\Procedure{Compute}{$\expr,p$}
\State Choose $N$ such that $N>2^{\frac{p+11}{3}}$
\State $p_x\gets p+2\lceil\log_2(\frac{N}{2}) \rceil+4$
\Repeat \label{algline:exp_go}
\State $\frac{m}{n}\gets \Call{Compute}{x,p_x}$\label{algline:exp}
\If{$0<\frac{m}{n}<1$} \Comment{Theorem~\ref{thm:base_exp}.\ref{thm:base_exp_1}}
	\State $\frac{m'}{n'}\gets (\frac{N}{N-\frac{2m}{n}})^{\frac{N}{2}}$
	\State \Return $\frac{m'}{n'}$
\ElsIf{$-1<\frac{m}{n}<0$} \Comment{Theorem~\ref{thm:base_exp}.\ref{thm:base_exp_2}}
	\State $\frac{m'}{n'}\gets \frac{1}{(\frac{N}{N+\frac{2m}{n}})^{\frac{N}{2}}}$
	\State \Return $\frac{m'}{n'}$
\Else
	\State Choose $k\in\nat$ such that $|\frac{m}{2^kn}|<1$
	\If{$(\frac{m}{n}>0) \wedge$ \Comment{Theorem~\ref{thm:extend_exp}.\ref{thm:extend_exp_1}}\\
		\hskip\algorithmicindent\hskip\algorithmicindent\phantom{\textbf{else}\textbf{if}}$(p_x-2\lceil\log_2(\frac{N}{2}) \rceil-2k-3\geq p)$}
		\State $\frac{m'}{n'}\gets (\frac{N}{N-\frac{2m}{n}})^{N\cdot 2^{k-1}}$
		\State \Return $\frac{m'}{n'}$
	\ElsIf{$(\frac{m}{n}<0) \wedge$ \Comment{Theorem~\ref{thm:extend_exp}.\ref{thm:extend_exp_2}}\\
	 \hskip\algorithmicindent\hskip\algorithmicindent\phantom{\textbf{else}\textbf{else if}}$(p_x-2\lceil\log_2(\frac{N}{2}) \rceil-2k-4\geq p)$}	
		\State $\frac{m'}{n'}\gets \frac{1}{(\frac{N}{N+\frac{2m}{n}})^{N\cdot 2^{k-1}}}$
		\State \Return $\frac{m'}{n'}$
	\Else
		\State $p_x\gets p_x +1$ \label{algline:exp_inc}
	\EndIf
\EndIf
\Until{$\true$}
\EndProcedure
\end{algorithmic}
\end{algorithm}
Algorithm~\ref{alg:exp} implements the approximations described by Theorem~\ref{thm:base_exp} and \ref{thm:extend_exp} to calculate $e^x$ with arbitrary precision. Observe that when $|\frac{m}{n}|<1$, $e^x$ can be approximated in one pass. However, when $|\frac{m}{n}|\geq 1$, loss of precision depends on the magnitude of $|\frac{m}{n}|$. Thus, recomputing $x$ with higher precisions might be necessary to compensate for the  loss of precision caused by applying equality~\eqref{eqn:exp_range} (see Line~\ref{algline:exp_go},\ref{algline:exp_inc} in Algorithm~\ref{alg:exp}).

To confirm that iterative computations are unavoidable, we apply perturbation analysis on $f(x)=e^x$:
\begin{align*}
|\frac{xf'(x)}{f(x)}|=|\frac{xe^x}{e^x}|=|x|
\end{align*}
The quantity $|x|$ can become arbitrarily large and hence approximating $e^x$ with precision $p$ in one pass is not always possible.

\subsection{Natural Logarithm}\label{subsec:log}
In this section, we first discuss an approximation for $\ln(x)$ based on Riemann sums where we assume that $x$ is precise. Then we extend this calculation to approximate $\ln(x)$ where $x$ is represented by $(m,n,p)$.

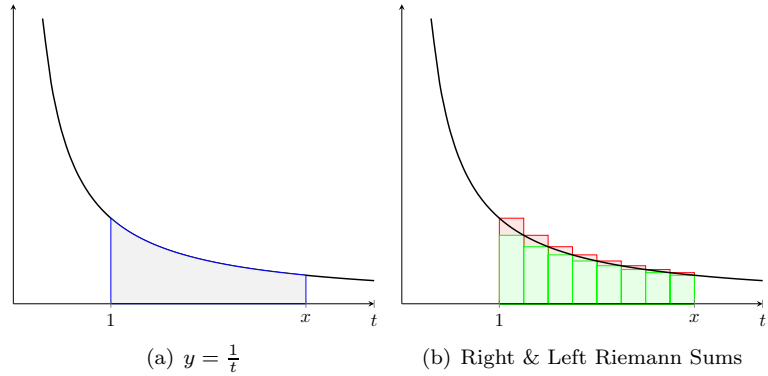
\begin{figure}[t]
\centering
\subfigure[$y=\frac{1}{t}$]{\label{fig:inverse}
\begin{tikzpicture}[scale=0.7]
\begin{axis}[
	ytick=\empty,
	xtick=\empty,
    xmax=3.7,ymax=3.5,ymin=0,xmin=0,
    enlargelimits=true,
    axis lines=middle,
    clip=false,
    every extra x tick/.style={
        grid=none, 
        tick0/.initial=red,
        tick1/.initial=green,
        tick2/.initial=orange,
        xticklabel style={
            anchor=north,
        },
    },
    extra x ticks ={1,3,3.7},
    extra x tick labels={$1$,$x$,$t$},   
    domain=0:4,
    axis on top
    ]
\addplot[smooth, thick,domain=0.3:3.7,samples=40]{x^-1};
\addplot+[mark=none,
        domain=1:3,
        samples=100,
        draw=blue,
        fill=gray!10
        ]{x^-1} \closedcycle; 
\end{axis}
\end{tikzpicture}}
\subfigure[Right \& Left Riemann Sums]{\label{fig:ln_riemann}
\begin{tikzpicture}[scale=0.7]
\begin{axis}[
	ytick=\empty,
	xtick=\empty,
    xmax=3.7,ymax=3.5,ymin=0,xmin=0,
    enlargelimits=true,
    axis lines=middle,
    clip=false,
    every extra x tick/.style={
        grid=none, 
        tick0/.initial=red,
        tick1/.initial=green,
        tick2/.initial=orange,
        xticklabel style={
            anchor=north,
        },
    },
    extra x ticks ={1,3,3.7},
    extra x tick labels={$1$,$x$,$t$},   
    domain=0:4,
    axis on top
    ]
\addplot [draw=red, fill=red!10, ybar interval, samples=9, domain=1:3]
    {x^-1}\closedcycle;
\addplot [draw=green, fill=green!10, ybar interval, samples=9, domain=3:1]
    {x^-1}\closedcycle;
\addplot[smooth, thick,domain=0.3:3.7,samples=40]{x^-1};
\end{axis}
\end{tikzpicture}}
\caption{Approximating $\ln(x)$}
\end{figure}

Suppose $x>1$ is a real number and we want to approximate $\ln(x)$. We consider the curve $y=\frac{1}{t}$ (see Fig.~\ref{fig:inverse}) and calculate the area enclosed by this curve and the t-axis between $t=1$ and $t=x$. This area can be calculated as follows:
\begin{align*}
\int_1^{x} \frac{dt}{t}=\ln(x)
\end{align*} 
We use Riemann sums to approximate this area; Fig.~\ref{fig:ln_riemann} shows how the area can be approximated from below and above using rectangles. Thus, we get the following inequalities for $N$ rectangles:
\begin{align*}
\lnsuml{x}{N} \leq \ln(x) \leq \lnsumu{x}{N}  
\end{align*}  
This gives us an upper bound and a lower bound for $\ln(x)$ and by increasing $N$ we get more precise approximations. 

We can estimate the precision of our approximations. For instance, if we approximate $\ln(x)$ by $\lnsumu{x}{N}$ the absolute error can be estimated as follows:
\begin{align}
&|\lnsumu{x}{N}-\ln(x)|\nonumber\\
&\phantom{\frac{x-1}{N}\sum}\leq|\lnsumu{x}{N}-\lnsuml{x}{N}|\nonumber\\
&\phantom{\frac{x-1}{N}\sum}= \frac{x-1}{N}(1-\frac{1}{x})=\frac{(x-1)^2}{Nx}\label{eqn:ln_riemann_prec}
\end{align}

Up to this point, we have assumed that the precise value of $x$ is available. In what follows, we formulate a theorem to describe an approximation of $\ln(x)$ based on a representation $(m,n,p)$ of $x$. 
\begin{mytheorem}\label{thm:ln}
Let $x$ be a real number represented by $(m,n,p)$ such that $p\geq 1$.
\begin{enumerate}[i.]
\item
If $\frac{m}{n}>1$, then $\ln(x)$ can be represented by $(m',n',p-j-4)$ where $\frac{m'}{n'}=\lnsumu{(\frac{m}{n})}{N}$, $N=\lceil 2^{p-2}\frac{(\frac{m}{n})^2}{\frac{m}{n}-1}\rceil$, and $j$ is the smallest natural number such that $j\geq \log_2(\frac{1+\frac{n}{m}}{1-\frac{n}{m}})$ holds. 
\item
If $0<\frac{m}{n}<1$, then $\ln(x)$ can be represented by $(m',n',p-j-5)$ where $\frac{m'}{n'}=-\lnsumu{(\frac{n}{m})}{N}$, $N=\lceil 2^{p-2}\frac{(\frac{n}{m})^2}{\frac{n}{m}-1}\rceil$, and $j$ is the smallest natural number such that $j\geq \log_2(\frac{1+\frac{m}{n}}{1-\frac{m}{n}})$ holds. 
\end{enumerate}

\end{mytheorem}

\begin{proof}
\textit{\\}
\begin{enumerate}[i.]
\item
Suppose $\frac{m}{n}>1$. The number $x$ is represented by $(m,n,p)$ and hence we can write:
\begin{align}
\frac{1}{2}<\frac{m}{2n}\leq \frac{m}{n}(1-\frac{1}{2^p})&<x<\frac{m}{n}(1+\frac{1}{2^p})\leq \frac{3m}{2n} \label{eqn:ln_x_bound}
\end{align}
To prove the theorem, we need to prove the following inequality:
\begin{align}
&|\ln(x)-\lnsumu{\frac{m}{n}}{N}|\nonumber\\
&\phantom{|\ln(x)-\frac{\frac{m}{n}-1}{N}\sum_{i=0}^{N-1}}<\frac{1}{2^{p-j-4}}|\lnsumu{\frac{m}{n}}{N}|\label{eqn:ln}
\end{align}
We rewrite the left hand side of inequality~\eqref{eqn:ln} as follows:
\begin{align*}
|\ln(x)-&\lnsumu{\frac{m}{n}}{N}|=\\
|\ln(x)-&\lnsumu{x}{N}\\
+&\lnsumu{x}{N}-\lnsumu{\frac{m}{n}}{N}|\leq \\
|\ln(x)-&\lnsumu{x}{N}|\\
+&|\lnsumu{x}{N}-\lnsumu{\frac{m}{n}}{N}|
\end{align*}
To prove inequality~\eqref{eqn:ln}, it suffices to show that the following inequalities hold:
\begin{align}
&|\ln(x)-\lnsumu{x}{N}|<\frac{1}{2^{p-j-3}}|\lnsumu{\frac{m}{n}}{N}|\label{eqn:ln_part_1} \\
&|\lnsumu{x}{N}-\lnsumu{\frac{m}{n}}{N}|\nonumber\\
&\phantom{|\lnsumu{x}{N}-\frac{\frac{m}{n}-1}{N}}<\frac{1}{2^{p-j-3}}|\lnsumu{\frac{m}{n}}{N}|\label{eqn:ln_part_2}
\end{align}
\textbf{Proof for inequality~\eqref{eqn:ln_part_1}:} 
We use inequality~\eqref{eqn:ln_riemann_prec} and calculate an upper bound for the left hand side of inequality~\eqref{eqn:ln_part_1}:
\begin{align}
|\ln(x)-\lnsumu{x}{N}|\leq \frac{(x-1)^2}{Nx}\label{eqn:ln_part_1_upper}
\end{align}
To calculate an upper bound for $\frac{(x-1)^2}{Nx}$, we consider the function $f(x)=\frac{(x-1)^2}{Nx}$ and calculate its derivative:
\begin{align*}
f'(x)=\frac{x^2-1}{Nx^2}
\end{align*}
From inequality~\eqref{eqn:ln_x_bound} we conclude that $x\in [\frac{1}{2},\frac{3m}{2n}]$. The function $f(x)$ is decreasing ($f'(x)\leq 0$) in the interval $[\frac{1}{2},1]$ and increasing ($f'(x)\geq 0$) in the interval $[1,\frac{3m}{2n}]$. Thus, the maximum of $f(x)$ for $x\in [\frac{1}{2},\frac{3m}{2n}]$ is $\max (f(\frac{1}{2}),f(\frac{3m}{2n}))$. We use this to rewrite inequality~\eqref{eqn:ln_part_1_upper}:
\begin{align}
|\ln(x)-\lnsumu{x}{N}|&\leq f(x)\leq \max (f(\frac{1}{2}),f(\frac{3m}{2n}))\nonumber \\ 
&=\max(\frac{1}{2N},\frac{(\frac{3m}{2n}-1)^2}{N(\frac{3m}{2n})})\nonumber\\
&\leq \max (\frac{1}{2N},\frac{(\frac{3m}{2n})^2}{N(\frac{3m}{2n})})=\frac{3m}{2Nn}\label{eqn:ln_part_1_upper_1}
\end{align}
We also calculate a lower bound for the right hand side of inequality~\eqref{eqn:ln_part_1}:
\begin{align}
\frac{1}{2^{p-j-3}}|\lnsumu{\frac{m}{n}}{N}|>&\frac{1}{2^{p-3}}|\frac{(\frac{m}{n}-1)}{N}\sum_{i=0}^{N-1}\frac{1}{1+(\frac{N-1}{N})(\frac{m}{n}-1)}|\nonumber\\
=&\frac{1}{2^{p-3}}.(\frac{m}{n}-1).\frac{N}{N+(N-1)(\frac{m}{n}-1)}\nonumber\\
>&\frac{1}{2^{p-3}}.(\frac{m}{n}-1).\frac{N}{N(1+\frac{m}{n}-1)}=\frac{(\frac{m}{n}-1)}{2^{p-3}(\frac{m}{n})}\label{eqn:ln_part_1_lower}
\end{align}
To show that inequality~\eqref{eqn:ln_part_1} holds, it suffices to prove the following inequality (see inequality~\eqref{eqn:ln_part_1_upper_1},\eqref{eqn:ln_part_1_lower}):
\begin{align*}
\frac{3m}{2Nn}<\frac{(\frac{m}{n}-1)}{2^{p-3}(\frac{m}{n})}
\end{align*}
Thus, it suffices to choose $N\geq 2^{p-2}\frac{(\frac{m}{n})^2}{\frac{m}{n}-1}$.

\textbf{Proof for inequality~\eqref{eqn:ln_part_2}:} To prove inequality~\eqref{eqn:ln_part_2} we consider the calculations in $\lnsumu{x}{N}$ and estimate the amount of precision that we lose in the approximation $\lnsumu{\frac{m}{n}}{N}$.

From Theorem~\ref{thm:add}.\ref{thm:add_diff_sign} we conclude that $j$ units of precision is lost by subtracting $1$ from $x$ where $j\geq \log_2(\frac{1+\frac{n}{m}}{1-\frac{n}{m}})$. We obtain the following inequalities:
\begin{gather*}
|(x-1)-(\frac{m}{n}-1)|<\frac{1}{2^{p-j}}|\frac{m}{n}-1|\\
|\frac{i}{N}(x-1)-\frac{i}{N}(\frac{m}{n}-1)|<\frac{1}{2^{p-j}}|\frac{i}{N}(\frac{m}{n}-1)|
\end{gather*}
The approximation $\frac{i}{N}(\frac{m}{n}-1)$ of $\frac{i}{N}(x-1)$ is positive. Hence, we do not lose precision by adding $\frac{i}{N}(x-1)$ and $1$ (see Theorem~\ref{thm:add}.\ref{thm:add_same_sign}).
\begin{align*}
|(1+\frac{i}{N}(x-1))-(1+\frac{i}{N}(\frac{m}{n}-1))|<\frac{1}{2^{p-j}}|1+\frac{i}{N}(\frac{m}{n}-1)|
\end{align*}
By approximating the inverse of $1+\frac{i}{N}(x-1)$, we lose $1$ unit of precision (see Theorem~\ref{thm:inverse}):
\begin{align*}
|\frac{1}{1+\frac{i}{N}(x-1)}-\frac{1}{1+\frac{i}{N}(\frac{m}{n}-1)}|<\frac{1}{2^{p-j-1}}|\frac{1}{1+\frac{i}{N}(\frac{m}{n}-1)}|
\end{align*}
We lose $2$ units of precision by multiplying $\frac{x-1}{N}$ and $\frac{1}{1+\frac{i}{N}(x-1)}$ (see Theorem~\ref{thm:mult}):
\begin{align*}
|\frac{x-1}{N}.\frac{1}{1+\frac{i}{N}(x-1)}-\frac{(\frac{m}{n}-1)}{N}.&\frac{1}{1+\frac{i}{N}(\frac{m}{n}-1)}|\\
&<\frac{1}{2^{p-j-3}}|\frac{(\frac{m}{n}-1)}{N}.\frac{1}{1+\frac{i}{N}(\frac{m}{n}-1)}|
\end{align*}
Since the numbers $\frac{x-1}{N}.\frac{1}{1+\frac{i}{N}(x-1)}$ are approximated by the positive numbers $\frac{\frac{m}{n}-1}{N}.\frac{1}{1+\frac{i}{N}(\frac{m}{n}-1)}$, calculating the summation $\lnsumu{x}{N}$ does not affect the precision:
\begin{align*}
|\lnsumu{x}{N}-&\lnsumu{\frac{m}{n}}{N}|\\
&<\frac{1}{2^{p-j-3}}|\lnsumu{\frac{m}{n}}{N}|
\end{align*}
\item
Suppose $0<\frac{m}{n}<1$. We use the following identity to approximate $\ln(x)$:
\begin{align*}
\ln(x)=-\ln(\frac{1}{x})
\end{align*}
We approximate $\frac{1}{x}$ by $(n,m,p-1)$ (see Theorem~\ref{thm:inverse}). Then, we apply the first part of the theorem and Theorem~\ref{thm:unary_neg} to approximate $-\ln(\frac{1}{x})$.
\end{enumerate}
\end{proof}

\begin{algorithm}
\caption{Natural Logarithm}\label{alg:log}
\begin{algorithmic}[1]
\Require $\expr$ has the shape $\ln(x)$
\Procedure{Compute}{$\expr,p$}
\State $p_x \gets p+5$
\Repeat \label{algline:ln_go}
\State $\frac{m}{n} \gets \Call{Compute}{x,p_x}$
\If{$\frac{m}{n}>1$} 
	\State $\textit{arg}\gets \frac{m}{n}$
	\State $\ell \gets 4$
\ElsIf{$0<\frac{m}{n}<1$}
	\State $\textit{arg}\gets \frac{n}{m}$
	\State $\ell \gets 5$
\Else
	\State ``Undefined operation''
\EndIf

\State $j\gets \lceil\log_2(\frac{1+\frac{1}{\textit{arg}}}{1-\frac{1}{\textit{arg}}}) \rceil$
\If{$p_x-j-\ell\geq p$}
	\State $N\gets \lceil 2^{p_x-2}\frac{\textit{arg}^2}{\textit{arg}-1}\rceil$
	\State $\frac{m'}{n'}\gets \lnsumu{\textit{arg}}{N}$
	\If{$\frac{m}{n}>1$}
		\State \Return $\frac{m'}{n'}$
	\Else
		\State \Return $\frac{-m'}{n'}$
	\EndIf
\Else
	\State $p_x\gets p_x+1$ \label{algline:ln_inc}
\EndIf
\Until{$\true$}
\EndProcedure
\end{algorithmic}
\end{algorithm}

Algorithm~\ref{alg:log} applies Theorem~\ref{thm:ln} to approximate $\ln(x)$ with arbitrary precision. Note that when the approximation $\frac{m}{n}$ is close to $1$ the amount of precision that we lose in the calculations depends on the magnitude of $\frac{m}{n}$. Loss of precision for $x\approx 1$ in Algorithm~\ref{alg:log} is due to our approximation formula, $\lnsumu{x}{N}$. We divide the interval between $1$ and $x$ into $N$ subintervals and approximate the area under the curve $f(t)=\frac{1}{t}$. The length of the interval $[1,x]$ is crucial in our approximation. Thus, recomputing $x$ with higher precisions is necessary when a significant amount of precision is lost in $x-1$ (see Line~\ref{algline:ln_go},\ref{algline:ln_inc} in Algorithm~\ref{alg:log}). 

To show that approximating $\ln(x)$ for $x\approx 1$ is fundamentally problematic we apply perturbation analysis  on $f(x)=\ln(x)$:  
\begin{align*}
|\frac{xf'(x)}{f(x)}|=|\frac{x(\frac{1}{x})}{\ln(x)}|=|\frac{1}{\ln(x)}|
\end{align*}
When $x\approx 1$ the quantity $\ln(x)$ is a small value and hence committing a small error in the approximation of $x$ causes a significant error in calculating $\ln(x)$. Thus, for $x\approx 1$, iterative computations  are unavoidable.

\subsection{Arctangent}\label{subsec:arctan}
We first introduce an approximation of $\arctan(x)$ using Riemann sums. The assumption is that the precise value of $x$ is available. Afterwards, we extend our calculations to introduce an approximation of $\arctan(x)$ based on a representation $(m,n,p)$ of $x$.

\begin{figure}[t]
\centering
\subfigure[$y=\frac{1}{1+t^2}$]{\label{fig:deriv_arctan}
\begin{tikzpicture}[scale=0.7]
\begin{axis}[
	ytick=\empty,
	xtick=\empty,
    xmax=3.5,ymax=2,ymin=0,xmin=0,
    enlargelimits=true,
    axis lines=middle,
    clip=false,
    every extra x tick/.style={
        grid=none, 
        tick0/.initial=red,
        tick1/.initial=green,
        tick2/.initial=orange,
        xticklabel style={
            anchor=north,
        },
    },
    extra x ticks ={3,3.5},
    extra x tick labels={$x$,$t$},   
    domain=0:3.5,
    axis on top
    ]
\addplot[smooth, thick,domain=0:3,samples=40]{1/(1+x^2)};
\addplot+[mark=none,
        domain=0:3,
        samples=100,
        draw=blue,
        fill=gray!10
        ]{1/(1+x^2)} \closedcycle; 
\end{axis}
\end{tikzpicture}}
\subfigure[Right \& Left Riemann Sums]{\label{fig:arctan_riemann}
\begin{tikzpicture}[scale=0.7]
\begin{axis}[
	ytick=\empty,
	xtick=\empty,
    xmax=3.5,ymax=2,ymin=0,xmin=0,
    enlargelimits=true,
    axis lines=middle,
    clip=false,
    every extra x tick/.style={
        grid=none, 
        tick0/.initial=red,
        tick1/.initial=green,
        tick2/.initial=orange,
        xticklabel style={
            anchor=north,
        },
    },
    extra x ticks ={3,3.5},
    extra x tick labels={$x$,$t$},   
    domain=0:3.5,
    axis on top
    ]

\addplot [draw=red, fill=red!10, ybar interval, samples=14, domain=0:3]
    {1/(1+x^2)}\closedcycle;
\addplot [draw=green, fill=green!10, ybar interval, samples=14, domain=3:0]
    {1/(1+x^2)}\closedcycle;

\addplot[smooth, thick,domain=0:3,samples=40]{1/(1+x^2)};

\end{axis}
\end{tikzpicture}}
\caption{Approximating $\arctan(x)$}
\end{figure}
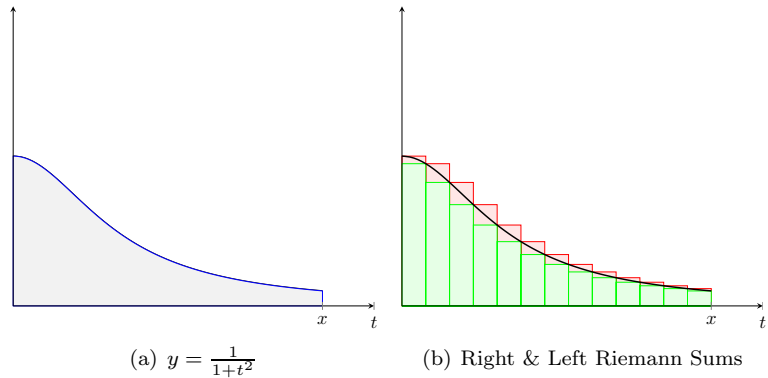
Suppose that $x>0$ is a real number and we want to approximate $\arctan(x)$. We consider the curve $y=\frac{1}{1+t^2}$; see Fig.~\ref{fig:deriv_arctan}. We calculate the area enclosed by this curve and the t-axis between $t=0$ and $t=x$:
\begin{align*}
\int_0^{x} \frac{dt}{1+t^2}=\arctan(x)
\end{align*}
We approximate this area using Riemann sums. Fig.~\ref{fig:arctan_riemann} shows approximations from above and below for the integral using rectangles. From Fig.~\ref{fig:arctan_riemann} we can derive the following inequalities for $N$ rectangles:
\begin{align*}
\arctansuml{x}{N} \leq \arctan(x) \leq \arctansumu{x}{N}
\end{align*}

We want to estimate the precision of our approximations. Suppose we approximate $\arctan(x)$ by $\arctansumu{x}{N}$. The absolute error of this approximation can be estimated as follows:
\begin{align}
|\arctansumu{x}{N}-\arctan(x)|&\leq |\arctansumu{x}{N}-\arctansuml{x}{N}|\nonumber\\
&=\frac{x}{N}(1-\frac{1}{1+x^2})=\frac{x^3}{N(1+x^2)}\label{eqn:arctan_riemann_prec}
\end{align}

Up to this point, we have assumed that $x$ is precisely calculated. In what follows, we assume that $x$ is approximated by $(m,n,p)$. We extend the Riemann sum calculation to compute $\arctan(x)$ using the given approximation of $x$. 
\begin{mytheorem}\label{thm:arctan}
Let $x$ be a real number represented by $(m,n,p)$. Then $\arctan(x)$ can be represented by $(m',n',p-6)$ where $\frac{m'}{n'}=\arctansumu{(\frac{m}{n})}{N}$ and $N= 2^{p-2}\lceil(\frac{m}{n})^2\rceil$.
\end{mytheorem}

\begin{proof}
We consider two cases:
\begin{enumerate}
\item
Suppose $\frac{m}{n}>0$. Since $x$ is represented by $(m,n,p)$ we can write:
\begin{align}
\frac{m}{n}(1-\frac{1}{2^p})<x<\frac{m}{n}(1+\frac{1}{2^p})\leq \frac{2m}{n}\label{eqn:arctan_x_bound}
\end{align}
To prove the theorem, we should show that:
\begin{align}
|\arctan(x)-\arctansumu{(\frac{m}{n})}{N}|<\frac{1}{2^{p-6}}|\arctansumu{(\frac{m}{n})}{N}|\label{eqn:arctan}
\end{align}
We rewrite the left hand side of inequality~\eqref{eqn:arctan} as follows:
\begin{align*}
&|\arctan(x)-\arctansumu{(\frac{m}{n})}{N}|=\\
&|\arctan(x)-\arctansumu{x}{N}\\
&\phantom{|\arctan(x)}+\arctansumu{x}{N}-\arctansumu{(\frac{m}{n})}{N}|\leq \\
&|\arctan(x)-\arctansumu{x}{N}|\\
&\phantom{|\arctan(x)}+|\arctansumu{x}{N}-\arctansumu{(\frac{m}{n})}{N}|
\end{align*}
Thus, to prove inequality~\eqref{eqn:arctan}, it suffices to prove the following inequalities:
\begin{align}
&|\arctan(x)-\arctansumu{x}{N}|<\frac{1}{2^{p-5}}|\arctansumu{(\frac{m}{n})}{N}|\label{eqn:arctan_part_1}\\
&|\arctansumu{x}{N}-\arctansumu{(\frac{m}{n})}{N}|\nonumber\\
&\phantom{|\arctan(x)-\arctansumu{x}{N}|}<\frac{1}{2^{p-5}}|\arctansumu{(\frac{m}{n})}{N}|\label{eqn:arctan_part_2}
\end{align}

\textbf{Proof for inequality~\eqref{eqn:arctan_part_1}:}
We first consider the left hand side of inequality~\eqref{eqn:arctan_part_1} and calculate an upper bound for it. From inequality~\eqref{eqn:arctan_riemann_prec} we can write:
\begin{align}
|\arctan(x)-\arctansumu{x}{N}|\leq \frac{x^3}{N(1+x^2)}\label{eqn:arctan_part_1_upper}
\end{align}
To obtain an upper-bound for $\frac{x^3}{N(1+x^2)}$, we consider the function $f(x)=\frac{x^3}{N(1+x^2)}$ and calculate its derivative:
\begin{align*}
f'(x)=\frac{3x^2+x^4}{N(1+x^2)^2}>0
\end{align*}
Thus, $f(x)$ is an increasing function and its maximum occurs when $x$ gets its maximum value. Inequality~\eqref{eqn:arctan_x_bound} implies that $\frac{2m}{n}$ is an upper bound for $x$ and hence we can rewrite inequality~\eqref{eqn:arctan_part_1_upper} as follows:
\begin{align}
|\arctan(x)-\arctansumu{x}{N}|\leq \frac{x^3}{N(1+x^2)}&\leq \frac{8(\frac{m}{n})^3}{N(1+4(\frac{m}{n})^2)}\nonumber\\
&<\frac{8(\frac{m}{n})^3}{N(1+(\frac{m}{n})^2)}\label{eqn:arctan_part_1_upper_1}
\end{align}
We also calculate a lower bound for the right hand side of inequality~\eqref{eqn:arctan_part_1}:
\begin{align}
\frac{1}{2^{p-5}}|\arctansumu{(\frac{m}{n})}{N}|>&(\frac{1}{2^{p-5}})\frac{(\frac{m}{n})}{N}\sum_{i=0}^{N-1}\frac{1}{1+(\frac{N-1}{N})^2(\frac{m}{n})^2}\nonumber\\
=&(\frac{1}{2^{p-5}}).(\frac{m}{n}).(\frac{N^2}{N^2+(N-1)^2(\frac{m}{n})^2})\nonumber\\
>&(\frac{1}{2^{p-5}}).(\frac{m}{n}).(\frac{N^2}{N^2(1+(\frac{m}{n})^2)})\nonumber\\
=&\frac{(\frac{m}{n})}{2^{p-5}(1+(\frac{m}{n})^2)}\label{eqn:arctan_part1_lower}
\end{align}
To prove inequality~\eqref{eqn:arctan_part_1}, it suffices to show that the following inequality holds (see inequality~\eqref{eqn:arctan_part_1_upper_1},\eqref{eqn:arctan_part1_lower}):
\begin{align*}
\frac{8(\frac{m}{n})^3}{N(1+(\frac{m}{n})^2)}<\frac{(\frac{m}{n})}{2^{p-5}(1+(\frac{m}{n})^2)}
\end{align*}
We divide all the components by $\frac{\frac{m}{n}}{(1+(\frac{m}{n})^2)}$, obtaining:
\begin{align*}
\frac{8(\frac{m}{n})^2}{N}<\frac{1}{2^{p-5}}
\end{align*}
Thus, it suffices to take $N= 2^{p-2}\lceil(\frac{m}{n})^2\rceil$.

\textbf{Proof for inequality~\eqref{eqn:arctan_part_2}:}
To prove the inequality, we consider the calculations in $\arctansumu{x}{N}$ and estimate the amount of error that we commit in the approximation $\arctansumu{(\frac{m}{n})}{N}$.

From Theorem~\ref{thm:mult} and inequality~\eqref{eqn:arctan_x_bound}, we conclude that $2$ units of precision is lost by approximating $x^2$ by $(\frac{m}{n})^2$ and hence we obtain:
\begin{gather*}
|x^2-(\frac{m}{n})^2|<\frac{1}{2^{p-2}}|(\frac{m}{n})^2|\\
|(\frac{i}{N})^2x^2-(\frac{i}{N})^2(\frac{m}{n})^2|<\frac{1}{2^{p-2}}|(\frac{i}{N})^2(\frac{m}{n})^2|
\end{gather*}
The approximation $(\frac{i}{N})^2(\frac{m}{n})^2$ of $(\frac{i}{N})^2x^2$ is positive. Thus, we do not lose precision by adding $(\frac{i}{N})^2x^2$ and $1$ (see Theorem~\ref{thm:add}.\ref{thm:add_same_sign}):
\begin{align*}
|(1+(\frac{i}{N})^2x^2)-(1+(\frac{i}{N})^2(\frac{m}{n})^2)|<\frac{1}{2^{p-2}}|1+(\frac{i}{N})^2(\frac{m}{n})^2|
\end{align*}
We lose $1$ unit of precision by approximating the inverse of $1+(\frac{i}{N})^2x^2$ (see Theorem~\ref{thm:inverse}):
\begin{align*}
|\frac{1}{1+(\frac{i}{N})^2x^2}-\frac{1}{1+(\frac{i}{N})^2(\frac{m}{n})^2}|<\frac{1}{2^{p-3}}|\frac{1}{1+(\frac{i}{N})^2(\frac{m}{n})^2}|
\end{align*}
We lose $2$ units of precision in the approximation of $\frac{x}{N}.\frac{1}{1+(\frac{i}{N})^2x^2}$ (see Theorem~\ref{thm:mult}):
\begin{align*}
|\frac{x}{N}.\frac{1}{1+(\frac{i}{N})^2x^2}-\frac{(\frac{m}{n})}{N}.\frac{1}{1+(\frac{i}{N})^2(\frac{m}{n})^2}|<\frac{1}{2^{p-5}}|\frac{(\frac{m}{n})}{N}.\frac{1}{1+(\frac{i}{N})^2(\frac{m}{n})^2}|
\end{align*}
The approximations $\frac{\frac{m}{n}}{N}.\frac{1}{1+(\frac{i}{N})^2(\frac{m}{n})^2}$ are positive for $0\leq i\leq N-1$. Thus, we do not lose precision by calculating the summation $\arctansumu{x}{N}$ (see Theorem~\ref{thm:add}.\ref{thm:add_same_sign}):
\begin{align*}
&|\arctansumu{x}{N}-\arctansumu{(\frac{m}{n})}{N}|\\
&\phantom{|\arctansumu{x}{N}-\frac{(\frac{m}{n})}{N}\sum \frac{m}{n}\frac{m}{n}}<\frac{1}{2^{p-5}}|\arctansumu{(\frac{m}{n})}{N}|
\end{align*}
\item
Suppose $\frac{m}{n}<0$. We can write:
\begin{align*}
\frac{m}{n}(1+\frac{1}{2^p})<x<\frac{m}{n}(1-\frac{1}{2^p})
\end{align*}
Observe that $x<0$; we can use the following identity for $\arctan(x)$:
\begin{align*}
\arctan(x)=-\arctan(-x)
\end{align*}
where $-x>0$. We first approximate $\arctan(-x)$ using the first part of the proof. Afterwards, we apply Theorem~\ref{thm:unary_neg} to approximate $-\arctan(-x)$. Unary negation does not influence the precision. 
\end{enumerate}
\end{proof}

\begin{algorithm}
\caption{Arctangent}\label{alg:arctan}
\begin{algorithmic}[1]
\Require $\expr$ has the shape $\arctan(x)$
\Procedure{Compute}{$\expr,p$}

\State $\frac{m}{n} \gets \Call{Compute}{x,p+6}$
\State $N\gets 2^{p+4}\lceil(\frac{m}{n})^2\rceil$
\State $\frac{m'}{n'}\gets \arctansumu{(\frac{m}{n})}{N}$
\State \Return $\frac{m'}{n'}$
\EndProcedure
\end{algorithmic}
\end{algorithm}

Algorithm~\ref{alg:arctan} applies Theorem~\ref{thm:arctan} to approximate $\arctan(x)$ with arbitrary precision. Note that  Theorem~\ref{thm:arctan} predicts the amount of precision that is lost by calculating $\arctan(x)$ independently of the argument $x$ and hence Algorithm~\ref{alg:arctan} calculates $\arctan(x)$ in one pass. 

To confirm that $\arctan(x)$ is computable in one pass, we apply perturbation analysis and calculate the quantity $|\frac{xf'(x)}{f(x)}|$ for $f(x)=\arctan(x)$:
\begin{align*}
|\frac{xf'(x)}{f(x)}|=|\frac{x(\frac{1}{1+x^2})}{\arctan(x)}|=|\frac{x}{(1+x^2)\arctan(x)}|
\end{align*}
Proposition~\ref{prop:arctan_condition} (see \ref{app:prop_lem}) shows that $|\frac{x}{(1+x^2)\arctan(x)}|<1$; iterative computations are not required for approximating $\arctan(x)$. 

%% file: transcendental_func_taylor.tex
\section{Approximating Transcendental Functions by Taylor Expansions}\label{sec:taylor}
In this section we first briefly discuss the basics of approximating functions using Taylor expansions. Afterwards, we use Taylor expansions to approximate $\sin(x)$ and $\cos(x)$.

Suppose $f: D\rightarrow R$ is a function  and $I=(a,b)\subseteq D$ such that:
\begin{itemize}
\item
$f$ has $n$ continuous derivatives on $I$ (denoted by $f^{(i)}(x)$ for $1\leq i\leq n$);
\item
$f^{(n+1)}$ exists on $I$;
\item
$x_0\in I$. 
\end{itemize}  Taylor's theorem states that for every $x\in I$
there is a number $c_x$ between $x$ and $x_0$ such that $f(x)=P_n(x)+R_n(x)$ where: 
\begin{align}
P_n(x)=\sum_{i=0}^n\frac{f^{(i)}(x_0)}{i!}(x-x_0)^i~,~R_n(x)=\frac{f^{(n+1)}(c_x)}{(n+1)!}(x-x_0)^{n+1}\label{eqn:taylor}
\end{align}
The formula $R_n(x)$ is called the Lagrange form of the remainder \cite{S67}. 

%In our approximations, we use Taylor expansions around the point $x_0=0$ and choose $I= (-1,1)$ as the base interval. In order to deal with an argument $x\in D\setminus I$, we use range reduction identities to bring the argument into the base interval. 
 
In the remaining of this section, we discuss approximations for $\sin(x)$ and $\cos(x)$. Our approximations are based on the following Taylor expansions around the point $x_0=0$: 
\begin{align}
\sin(x)=&\taylorsin{0}{\infty}{x} \label{eqn:taylor_sin}\\
\cos(x)=&\taylorcos{0}{\infty}{x}\label{eqn:taylor_cos}
\end{align}

For each function, we first describe an approximation that is applicable to the base interval $I=(-1,1)$ (see Section~\ref{subsec:base}). Afterwards, we extend our calculations to the complete domain of the functions using range reduction identities (see Section~\ref{subsec:extend}). Proof of correctness for the base and general cases are provided. Moreover, the \textproc{Compute}$(expr,p)$ function is extended to approximate $\sin(x)$ and $\cos(x)$ with arbitrary precision. We use perturbation analysis to confirm the observations obtained from the approximations.

\subsection{Approximating Functions in the Base Interval}\label{subsec:base}
\subsubsection{Sine}\label{subsubsec:base_sin}
In this section we use the Taylor expansion from equality~\eqref{eqn:taylor_sin} to approximate $\sin(x)$. We assume that the input argument $x$ is represented by $(m,n,p)$ and $|\frac{m}{n}|<1$.
\begin{mytheorem}\label{thm:base_sin}
Let $x$ be a real number represented by $(m,n,p)$ such that $-1<\frac{m}{n}<1$. Then $\sin(x)$ can be represented by $(m',n',p-2\lceil\log_2(2N-1)\rceil-3)$ where $\frac{m'}{n'}=\taylorsin{0}{N}{(\frac{m}{n})}$ and $N\in\nat$ is an odd number such that $(\frac{5}{6})(\frac{(2N+2)!(2N-1)^2}{2^{2N}})>2^p$. 
\end{mytheorem}

\begin{proof}
The real number $x$ is given with precision $p$:
\begin{align}
\frac{m}{n}-|\frac{m}{n}|\leq \frac{m}{n}-\frac{1}{2^p}|\frac{m}{n}|<x<\frac{m}{n}+\frac{1}{2^p}|\frac{m}{n}|\leq \frac{m}{n}+|\frac{m}{n}|\label{eqn:sin_bound}
\end{align}
To prove the theorem, we should show that the following inequality holds:
\begin{align}
|\sin(x)-\taylorsin{0}{N}{(\frac{m}{n})}|<\frac{1}{2^{p-2\lceil\log_2(2N-1)\rceil-3}}|\taylorsin{0}{N}{(\frac{m}{n})}|\label{eqn:sin}
\end{align}
We rewrite the left hand side of inequality~\eqref{eqn:sin} as follows:
\begin{align*}
&|\sin(x)-\taylorsin{0}{N}{(\frac{m}{n})}|=\\
&|\sin(x)-\taylorsin{0}{N}{x}+\taylorsin{0}{N}{x}-\taylorsin{0}{N}{(\frac{m}{n})}|\leq \\
&|\sin(x)-\taylorsin{0}{N}{x}|+|\taylorsin{0}{N}{x}-\taylorsin{0}{N}{(\frac{m}{n})}|
\end{align*}
Thus, we prove inequality~\eqref{eqn:sin} by showing that the following inequalities hold:
\begin{align}
&|\sin(x)-\taylorsin{0}{N}{x}|<\frac{1}{2^{p-2\lceil\log_2(2N-1)\rceil-2}}|\taylorsin{0}{N}{(\frac{m}{n})}|\label{eqn:sin_part_1}\\
&|\taylorsin{0}{N}{x}-\taylorsin{0}{N}{(\frac{m}{n})}|\nonumber\\
&\phantom{|\sin(x)-\taylorsin{0}{N}{x}|}<\frac{1}{2^{p-2\lceil\log_2(2N-1)\rceil-2}}|\taylorsin{0}{N}{(\frac{m}{n})}|\label{eqn:sin_part_2}
\end{align}

\noindent\textbf{Proof for inequality~\eqref{eqn:sin_part_1}:} Based on Taylor's theorem (see equality~\eqref{eqn:taylor}) and inequality~\eqref{eqn:sin_bound}, we can rewrite the left hand side of inequality~\eqref{eqn:sin_part_1} as follows:
\begin{align}
|\sin(x)-\taylorsin{0}{N}{x}|=|\frac{\sin^{(2N+2)}(c_x)x^{2N+2}}{(2N+2)!}|\leq \frac{|\frac{m}{n}|^{2N+2}2^{2N+2}}{(2N+2)!}\label{eqn:sin_part_1_left}
\end{align}
We choose $N=2k+1\geq 1$. Applying Proposition~\ref{prop:sin_sign} (see \ref{app:prop_lem}) we can rewrite the right hand side of inequality~\eqref{eqn:sin_part_1} as follows:
\begin{align}
\frac{1}{2^{p-2\lceil\log_2(2N-1)\rceil-2}}|\taylorsin{0}{N}{(\frac{m}{n})}|=&\bigg(\frac{sgn(\frac{m}{n})}{2^{p-2\lceil\log_2(2N-1)\rceil-2}}\bigg)\nonumber\\
\cdot &\bigg(\taylorsin{0}{N}{(\frac{m}{n})}\bigg)\label{eqn:sin_part_1_right}
\end{align} 
where $sgn$ is the sign function. To prove inequality~\eqref{eqn:sin_part_1}, it suffices to show that the following inequality holds (see \eqref{eqn:sin_part_1_left},\eqref{eqn:sin_part_1_right}):
\begin{align*}
\frac{|\frac{m}{n}|^{2N+2}2^{2N+2}}{(2N+2)!}<\frac{sgn(\frac{m}{n})}{2^{p-2\lceil\log_2(2N-1)\rceil-2}}\taylorsin{0}{N}{(\frac{m}{n})}
\end{align*}
This is equivalent to:
\begin{align}
&\frac{sgn(\frac{m}{n})}{2^{p-2\lceil\log_2(2N-1)\rceil-2}}\taylorsin{0}{N}{(\frac{m}{n})}-\frac{|\frac{m}{n}|^{2N+2}2^{2N+2}}{(2N+2)!}=\nonumber\\
&\frac{sgn(\frac{m}{n})}{2^{p-2\lceil\log_2(2N-1)\rceil-2}}\bigg(\frac{m}{n}-(\frac{1}{3!})(\frac{m}{n})^3+\taylorsin{2}{N}{(\frac{m}{n})}\bigg)\nonumber\\
&\phantom{\frac{sgn(\frac{m}{n})}{2^{p-2\lceil\log_2(2N-1)\rceil-2}}\bigg(\frac{m}{n}-(\frac{1}{3!})(\frac{m}{n})^3}-\frac{|\frac{m}{n}|^{2N+2}2^{2N+2}}{(2N+2)!}>0\label{eqn:sin_part_1_rewrite}
\end{align}
The quantity $sgn(\frac{m}{n})\taylorsin{2}{N}{(\frac{m}{n})}$ is positive (see Proposition~\ref{prop:sin_sign} in \ref{app:prop_lem}). Moreover, the approximation $\frac{m}{n}$ satisfies $|\frac{m}{n}|<1$. To prove inequality~\eqref{eqn:sin_part_1_rewrite} it is sufficient to show:
\begin{align}
\frac{sgn(\frac{m}{n})}{2^{p-2\lceil\log_2(2N-1)\rceil-2}}(\frac{m}{n}-(\frac{1}{3!})(\frac{m}{n})^3)-\frac{|\frac{m}{n}|^{2N+2}2^{2N+2}}{(2N+2)!}>\nonumber\\
\frac{sgn(\frac{m}{n})(2N-1)^2}{2^{p-2}}(\frac{m}{n}-(\frac{1}{3!})(\frac{m}{n})^3)-\frac{|\frac{m}{n}|^{3}2^{2N+2}}{(2N+2)!}>0\label{eqn:sin_part_1_rewrite_1}
\end{align}
Inequality~\eqref{eqn:sin_part_1_rewrite_1} is equivalent to: 
\begin{align*}
&sgn(\frac{m}{n})(\frac{m}{n})(1-(\frac{1}{3!})(\frac{m}{n})^2-\frac{2^{p+2N}}{(2N+2)!(2N-1)^2}(\frac{m}{n})^2)>0
\end{align*}
Since $sgn(\frac{m}{n})(\frac{m}{n})>0$ and $(\frac{m}{n})^2<1$, we should choose an $N$ such that:
\begin{align*}
\frac{1}{6}+\frac{2^{p+2N-2\lceil\log_2(2N-1) \rceil}}{(2N+2)!(2N-1)^2}<1
\end{align*}
Thus, $N$ should satisfy $(\frac{5}{6})\frac{(2N+2)!(2N-1)^2}{2^{2N}}>2^p$.

\noindent\textbf{Proof for inequality~\eqref{eqn:sin_part_2}:} We show that given an approximation $\frac{m}{n}$ of $x$ with precision $p$ we can approximate $\taylorsin{0}{N}{x}$ by $\taylorsin{0}{N}{(\frac{m}{n})}$ with precision $p-2\lceil\log_2(2N-1)\rceil -2$.  

The sign of the terms $\frac{(-1)^i}{(2i+1)!}x^{2i+1}$ alternates between positive and negative. Thus, adding two arbitrary terms with different signs from the expansion can significantly reduce the precision (see Theorem~\ref{thm:add}.\ref{thm:add_diff_sign}). To avoid this, we first consider specific pairs of terms for which loss of precision due to addition is bounded. Then, we calculate the summation of these pairs. The following identity shows the way we calculate $\taylorsin{0}{N}{x}$:
\begin{align*}
\taylorsin{0}{N}{x}=\sum_{i=0}^{k}\frac{x^{4i+1}}{(4i+1)!}(1-\frac{x^2}{(4i+2)(4i+3)})
\end{align*}
Choosing $N=2k+1$ allows us to make pairs of terms. 

The number $x$ is given with precision $p$. Thus, we can approximate $x^2$ with precision $p-2$ (see Theorem~\ref{thm:mult}):
\begin{align*}
|x^2-(\frac{m}{n})^2|<\frac{1}{2^{p-2}}|(\frac{m}{n})^2|
\end{align*}
and hence
\begin{align*}
|\frac{x^2}{(4i+2)(4i+3)}-\frac{(\frac{m}{n})^2}{(4i+2)(4i+3)}|<\frac{1}{2^{p-2}}|\frac{(\frac{m}{n})^2}{(4i+2)(4i+3)}|
\end{align*}
We approximate $1-\frac{x^2}{(4i+2)(4i+3)}$ by $1-\frac{(\frac{m}{n})^2}{(4i+2)(4i+3)}$. Loss of precision in the approximation can be estimated by calculating the quantity $\log_2(\frac{1+\frac{(\frac{m}{n})^2}{(4i+2)(4i+3)}}{1-\frac{(\frac{m}{n})^2}{(4i+2)(4i+3)}})$ (see Theorem~\ref{thm:add}.\ref{thm:add_diff_sign}):
\begin{align*}
\log_2(\frac{1+\frac{(\frac{m}{n})^2}{(4i+2)(4i+3)}}{1-\frac{(\frac{m}{n})^2}{(4i+2)(4i+3)}})\leq \log_2(\frac{1+\frac{1}{6}}{1-\frac{1}{6}})=\log_2(\frac{7}{5})<1
\end{align*}
Thus, we lose at most $1$ unit of precision:
\begin{align*}
|(1-\frac{x^2}{(4i+2)(4i+3)})-(1-\frac{(\frac{m}{n})^2}{(4i+2)(4i+3)})|<\frac{1}{2^{p-3}}|1-\frac{(\frac{m}{n})^2}{(4i+2)(4i+3)}|
\end{align*}
The powers $x^{4i+1}$ are approximated by $(\frac{m}{n})^{4i+1}$ for $0\leq i\leq k$ and $2\lceil\log_2(4i+1) \rceil$ units of precision is lost in this operation. Thus, in the worst case, the precision is reduced by $2\lceil\log_2(4k+1) \rceil$ units: 
\begin{align*}
|x^{4i+1}-(\frac{m}{n})^{4i+1}|<\frac{1}{2^{p-2\lceil\log_2(4k+1)\rceil}}|(\frac{m}{n})^{4i+1}|=\frac{1}{2^{p-2\lceil\log_2(2N-1)\rceil}}|(\frac{m}{n})^{4i+1}|
\end{align*}
Multiplying approximations of $x^{4i+1}$ by a constant does not influence the precision:
\begin{align*}
|\frac{x^{4i+1}}{(4i+1)!}-\frac{(\frac{m}{n})^{4i+1}}{(4i+1)!}|<\frac{1}{2^{p-2\lceil\log_2(2N-1)\rceil}}| \frac{(\frac{m}{n})^{4i+1}}{(4i+1)!}|
\end{align*}
The multiplication $\frac{x^{4i+1}}{(4i+1)!}(1-\frac{x^2}{(4i+2)(4i+3)})$ can be approximated by multiplying the approximations $\frac{(\frac{m}{n})^{4i+1}}{(4i+1)!}$ and $1-\frac{(\frac{m}{n})^2}{(4i+2)(4i+3)}$ (see Theorem~\ref{thm:mult}):
\begin{align*}
|\frac{x^{4i+1}}{(4i+1)!}(1&-\frac{x^2}{(4i+2)(4i+3)})-\frac{(\frac{m}{n})^{4i+1}}{(4i+1)!}(1-\frac{(\frac{m}{n})^2}{(4i+2)(4i+3)})|\\
&<\frac{1}{2^{p-\max\{2\lceil\log_2(2N-1) \rceil,3\}-2}}|\frac{(\frac{m}{n})^{4i+1}}{(4i+1)!}(1-\frac{(\frac{m}{n})^2}{(4i+2)(4i+3)})|\\
&=\frac{1}{2^{p-2\lceil\log_2(2N-1) \rceil-2}}|\frac{(\frac{m}{n})^{4i+1}}{(4i+1)!}(1-\frac{(\frac{m}{n})^2}{(4i+2)(4i+3)})|
\end{align*}
For the last equality, we use the assumptions $(\frac{5}{6})(\frac{(2N+2)!(N-1)^2}{2^{2N-2\lceil\log_2(2N-1) \rceil}})>2^p$ and $N=2k+1$; we can conclude that $N\geq 3$.

For $0\leq i\leq k$ the terms $\frac{(\frac{m}{n})^{4i+1}}{(4i+1)!}(1-\frac{(\frac{m}{n})^2}{(4i+2)(4i+3)})$ have the same sign which is determined by the sign of $\frac{m}{n}$ (see Proposition~\ref{prop:sin_sign} in \ref{app:prop_lem}). Thus, we can approximate $\sum_{i=0}^{k}\frac{x^{4i+1}}{(4i+1)!}(1-\frac{x^2}{(4i+2)(4i+3)})$ by $\sum_{i=0}^{k}\frac{(\frac{m}{n})^{4i+1}}{(4i+1)!}(1-\frac{(\frac{m}{n})^2}{(4i+2)(4i+3)})$ without losing precision (see Theorem~\ref{thm:add}.\ref{thm:add_same_sign}):
\begin{align*}
|\sum_{i=0}^{k}\frac{x^{4i+1}}{(4i+1)!}(1-\frac{x^2}{(4i+2)(4i+3)})-\sum_{i=0}^{k}\frac{(\frac{m}{n})^{4i+1}}{(4i+1)!}(1-\frac{(\frac{m}{n})^2}{(4i+2)(4i+3)})|<\\
\frac{1}{2^{p-2\lceil\log_2(2N-1)\rceil -2}}|\sum_{i=0}^{k}\frac{(\frac{m}{n})^{4i+1}}{(4i+1)!}(1-\frac{(\frac{m}{n})^2}{(4i+2)(4i+3)})|
\end{align*}
\end{proof}

Theorem~\ref{thm:base_sin} provides a top-down approximation for $\sin(x)$ in the base interval. Loss of precision in this approximation is estimated independently of the argument $x$. To show that iterative calculations can be avoided in the base interval, we calculate $|\frac{xf'(x)}{f(x)}|$ for $f(x)=\sin(x)$:
\begin{align}
|\frac{xf'(x)}{f(x)}|=|\frac{x\cdot \cos(x)}{\sin(x)}|=|x\cdot \cot(x)|\label{eqn:pert_sin}
\end{align}
Proposition~\ref{prop:sin_condition} (see \ref{app:prop_lem}) shows that $|x\cdot \cot(x)|<1$ for $x\in (-1,1)$. Thus, we can approximate $\sin(x)$ in the base interval in one pass. 
\subsubsection{Cosine}
\label{subsubsec:base_cos}
In this section we introduce an approximation for $\cos(x)$ where $x$ is represented by $(m,n,p)$ and $|\frac{m}{n}|<1$. Our approximation is based on the Taylor expansion from equality~\eqref{eqn:taylor_cos}.
\begin{mytheorem}\label{thm:base_cos}
Let $x$ be a real number represented by $(m,n,p)$ such that $-1<\frac{m}{n}<1$. Then $\cos(x)$ can be represented by $(m',n',p-2\lceil \log_2(2N-2) \rceil -3)$ where $\frac{m'}{n'}=\taylorcos{0}{N}{(\frac{m}{n})}$ and $N\in\nat$ is an odd number such that $\frac{(2N+1)!(2N-2)^2}{2^{2N}}>2^p$.
\end{mytheorem}

\begin{proof}
The real number $x$ is represented by $(m,n,p)$ and hence we can write:
\begin{align}
\frac{m}{n}-|\frac{m}{n}|<\frac{m}{n}-\frac{1}{2^p}|\frac{m}{n}|<x<\frac{m}{n}+\frac{1}{2^p}|\frac{m}{n}|\leq \frac{m}{n}+|\frac{m}{n}|\label{eqn:cos_bound}
\end{align}
To prove the theorem, we need to show:
\begin{align}
|\cos(x)-\taylorcos{0}{N}{(\frac{m}{n})}|<\frac{1}{2^{p-2\lceil\log_2(2N-2)\rceil -3}}|\taylorcos{0}{N}{(\frac{m}{n})}|\label{eqn:cos}
\end{align}
We rewrite the left hand side of inequality~\eqref{eqn:cos} as follows:
\begin{align*}
&|\cos(x)-\taylorcos{0}{N}{(\frac{m}{n})}|=\\
&|\cos(x)-\taylorcos{0}{N}{x}+\taylorcos{0}{N}{x}-\taylorcos{0}{N}{(\frac{m}{n})}|\leq\\
&|\cos(x)-\taylorcos{0}{N}{x}|+|\taylorcos{0}{N}{x}-\taylorcos{0}{N}{(\frac{m}{n})}|
\end{align*}
We prove inequality~\eqref{eqn:cos} by showing that the following inequalities are valid:
\begin{gather}
|\cos(x)-\taylorcos{0}{N}{x}|<\frac{1}{2^{p-2\lceil\log_2(2N-2)\rceil -2}}|\taylorcos{0}{N}{(\frac{m}{n})}|\label{eqn:cos_part_1}\\
|\taylorcos{0}{N}{x}-\taylorcos{0}{N}{(\frac{m}{n})}|<\frac{1}{2^{p-2\lceil\log_2(2N-2)\rceil -2}}|\taylorcos{0}{N}{(\frac{m}{n})}|\label{eqn:cos_part_2}
\end{gather}

\noindent \textbf{Proof for inequality~\eqref{eqn:cos_part_1}:} We use Taylor's theorem (see equality~\eqref{eqn:taylor}) and the bounds calculated for $x$  in inequality~\eqref{eqn:cos_bound} to rewrite the left hand side of inequality~\eqref{eqn:cos_part_1}:
\begin{align}
|\cos(x)-\taylorcos{0}{N}{x}|=|\frac{\cos^{(2N+1)}(c_x)x^{2N+1}}{(2N+1)!}|\leq \frac{|\frac{m}{n}|^{2N+1}2^{2N+1}}{(2N+1)!}\label{eqn:cos_part_1_left}
\end{align}
We choose $N=2k+1\geq 1$ and apply Proposition~\ref{prop:cos_sign} (see \ref{app:prop_lem}) to rewrite the right hand side of inequality~\eqref{eqn:cos_part_1}:
\begin{align}
\frac{1}{2^{p-2\lceil\log_2(2N-2)\rceil -2}}|\taylorcos{0}{N}{(\frac{m}{n})}|=\frac{1}{2^{p-2\lceil\log_2(2N-2)\rceil -2}}\taylorcos{0}{N}{(\frac{m}{n})}\label{eqn:cos_part_1_right}
\end{align}
To show that inequality~\eqref{eqn:cos_part_1} holds, it suffices to prove the following (see inequality~\eqref{eqn:cos_part_1_left},\eqref{eqn:cos_part_1_right}):
\begin{align}
\frac{|\frac{m}{n}|^{2N+1}2^{2N+1}}{(2N+1)!}<\frac{1}{2^{p-2\lceil\log_2(2N-2)\rceil -2}}\taylorcos{0}{N}{(\frac{m}{n})}\label{eqn:cos_part_1_rewrite}
\end{align}
Inequality~\eqref{eqn:cos_part_1_rewrite} is equivalent to:
\begin{align*}
&\frac{1}{2^{p-2\lceil\log_2(2N-2)\rceil -2}}\taylorcos{0}{N}{(\frac{m}{n})}-\frac{|\frac{m}{n}|^{2N+1}2^{2N+1}}{(2N+1)!}=\\
&\frac{1}{2^{p-2\lceil\log_2(2N-2)\rceil -2}}(1-\frac{1}{2}(\frac{m}{n})^2+\taylorcos{2}{N}{(\frac{m}{n})})-\frac{|\frac{m}{n}|^{2N+1}2^{2N+1}}{(2N+1)!}>0
\end{align*}
From Proposition~\ref{prop:cos_sign} (see \ref{app:prop_lem}) we conclude that the quantity $\taylorcos{2}{N}{(\frac{m}{n})}$ is positive. Since $|\frac{m}{n}|<1$, it suffices to show:
\begin{align}
&\frac{1}{2^{p-2\lceil\log_2(2N-2)\rceil -2}}(1-\frac{1}{2}(\frac{m}{n})^2)-\frac{|\frac{m}{n}|^{2N+1}2^{2N+1}}{(2N+1)!}>\nonumber \\
&\frac{(2N-2)^2}{2^{p -2}}(1-\frac{1}{2}(\frac{m}{n})^2)-\frac{2^{2N+1}}{(2N+1)!}(\frac{m}{n})^2>0\label{eqn:cos_part_1_rewrite1}
\end{align}
Inequality~\eqref{eqn:cos_part_1_rewrite1} is equivalent to:
\begin{align*}
1-\frac{1}{2}(\frac{m}{n})^2-\frac{2^{p+2N-1}}{(2N+1)!(2N-2)^2}(\frac{m}{n})^2>0
\end{align*}
Since $(\frac{m}{n})^2<1$, it is sufficient to choose an $N$ that satisfies the following inequality:
\begin{align*}
\frac{1}{2}+\frac{2^{p+2N-1}}{(2N+1)!(2N-2)^2}<1
\end{align*}
Thus, we should choose an $N=2k+1$ that satisfies $\frac{(2N+1)!(2N-2)^2}{2^{2N}}>2^p$.

\noindent\textbf{Proof for inequality~\eqref{eqn:cos_part_2}:} To prove the inequality, we estimate the amount of precision that is lost when we approximate $\taylorcos{0}{N}{x}$ by $\taylorcos{0}{N}{(\frac{m}{n})}$. The sign of the terms $\frac{(-1)^i}{(2i)!}x^{2i}$ alternates between positive and negative. Thus, adding two arbitrary terms with different signs from the expansion can potentially cause significant loss of precision (see Theorem~\ref{thm:add}.\ref{thm:add_diff_sign}). To avoid this, we first consider pairs of terms for which addition can be calculated with a bounded loss of precision. Afterwards, we calculate the summation of these pairs. The following identity shows our computation scheme:
\begin{align*}
\taylorcos{0}{N}{x}=\sum_{i=0}^{k}\frac{x^{4i}}{(4i)!}(1-\frac{x^2}{(4i+1)(4i+2)})
\end{align*}
Choosing $N=2k+1$ allows us to pair the terms of the summation. 

Since $x$ is given with precision $p$, we can approximate $x^2$ with precision $p-2$ (see Theorem~\ref{thm:mult}): 
\begin{align*}
|x^2-(\frac{m}{n})^2|<\frac{1}{2^{p-2}}|\frac{m}{n}|^2
\end{align*}
Multiplying the approximation of $x^2$ by a constant does not influence the precision:
\begin{align*}
|\frac{x^2}{(4i+1)(4i+2)}-\frac{(\frac{m}{n})^2}{(4i+1)(4i+2)}|<\frac{1}{2^{p-2}}|\frac{(\frac{m}{n})^2}{(4i+1)(4i+2)}|
\end{align*}
We approximate $1-\frac{x^2}{(4i+1)(4i+2)}$ by $1-\frac{(\frac{m}{n})^2}{(4i+1)(4i+2)}$. To estimate lose of precision in our approximation, we calculate the quantity $\log_2(\frac{1+\frac{(\frac{m}{n})^2}{(4i+1)(4i+2)}}{1-\frac{(\frac{m}{n})^2}{(4i+1)(4i+2)}})$ (see Theorem~\ref{thm:add}.\ref{thm:add_diff_sign}):
\begin{align*}
\log_2(\frac{1+\frac{(\frac{m}{n})^2}{(4i+1)(4i+2)}}{1-\frac{(\frac{m}{n})^2}{(4i+1)(4i+2)}})\leq \log_2(\frac{1+\frac{1}{2}}{1-\frac{1}{2}})=\log_2 3<2
\end{align*}
Thus, we lose at most $2$ units of precision:
\begin{align*}
|(1-\frac{x^2}{(4i+1)(4i+2)})-(1-\frac{(\frac{m}{n})^2}{(4i+1)(4i+2)})|<\frac{1}{2^{p-4}}|1-\frac{(\frac{m}{n})^2}{(4i+1)(4i+2)}|
\end{align*}
Approximating $x^{4i}$ by $(\frac{m}{n})^{4i}$ reduces the precision by $2\lceil\log_2 (4i) \rceil$ units (see Lemma~\ref{lm:power} in \ref{app:prop_lem}); in the worst case we lose $2\lceil\log_2 (4k) \rceil$ units of precision:
\begin{align*}
|x^{4i}-(\frac{m}{n})^{4i}|<\frac{1}{2^{p-2\lceil\log_2 (4k) \rceil}}|(\frac{m}{n})^{4i}|=\frac{1}{2^{p-2\lceil\log_2 (2N-2) \rceil}}|(\frac{m}{n})^{4i}|
\end{align*}
Multiplying $(\frac{m}{n})^{4i}$ by a constant factor does not influence the precision of the calculation: 
\begin{align*}
|\frac{x^{4i}}{(4i)!}-\frac{(\frac{m}{n})^{4i}}{(4i)!}|<\frac{1}{2^{p-2\lceil\log_2 (2N-2) \rceil}}|\frac{(\frac{m}{n})^{4i}}{(4i)!}|
\end{align*}
We approximate $\frac{x^{4i}}{(4i)!}(1-\frac{x^2}{(4i+1)(4i+2)})$ by multiplying the approximations $\frac{(\frac{m}{n})^{4i}}{(4i)!}$ and $(1-\frac{(\frac{m}{n})^2}{(4i+1)(4i+2)})$ (see Theorem~\ref{thm:mult}):
\begin{align*}
|\frac{x^{4i}}{(4i)!}(1-&\frac{x^2}{(4i+1)(4i+2)})-\frac{(\frac{m}{n})^{4i}}{(4i)!}(1-\frac{(\frac{m}{n})^2}{(4i+1)(4i+2)})|\\
~~~~~<&\frac{1}{2^{p-\max\{4,2\lceil\log_2(2N-2)\rceil\}-2}}|\frac{(\frac{m}{n})^{4i}}{(4i)!}(1-\frac{(\frac{m}{n})^2}{(4i+1)(4i+2)})|\\
~~~~~= &\frac{1}{2^{p-2\lceil\log_2(2N-2)\rceil-2}}|\frac{(\frac{m}{n})^{4i}}{(4i)!}(1-\frac{(\frac{m}{n})^2}{(4i+1)(4i+2)})|
\end{align*}
To obtain the last equality, we use the assumptions $\frac{(2N+1)!(2N-2)^2}{2^{2N}}>2^p$ and $N=2k+1$; we conclude that $N\geq 3$.

The terms $\frac{(\frac{m}{n})^{4i}}{(4i)!}(1-\frac{(\frac{m}{n})^2}{(4i+1)(4i+2)})$ are positive for $0\leq i\leq k$ (see Proposition~\ref{prop:cos_sign} in \ref{app:prop_lem}) and hence we can approximate the summation $\sum_{i=0}^{k}\frac{x^{4i}}{(4i)!}(1-\frac{x^2}{(4i+1)(4i+2)})$ without losing precision (see Theorem~\ref{thm:add}.\ref{thm:add_same_sign}):
\begin{align*}
|\sum_{i=0}^{k}\frac{x^{4i}}{(4i)!}(1-&\frac{x^2}{(4i+1)(4i+2)})-\sum_{i=0}^{k}\frac{(\frac{m}{n})^{4i}}{(4i)!}(1-\frac{(\frac{m}{n})^2}{(4i+1)(4i+2)})|<\\
~~~~~&\frac{1}{2^{p-2\lceil\log_2(2N-2) \rceil -2}}|\sum_{i=0}^{k}\frac{(\frac{m}{n})^{4i}}{(4i)!}(1-\frac{(\frac{m}{n})^2}{(4i+1)(4i+2)})|
\end{align*}
\end{proof}

Theorem~\ref{thm:base_cos} estimates loss of precision in $\cos(x)$ in the base interval independently of the argument $x$. To show that iterative computations can be avoided in the base interval, we calculate $|\frac{xf'(x)}{f(x)}|$ for $f(x)=\cos(x)$:
\begin{align}
|\frac{xf'(x)}{f(x)}|=|\frac{-x\cdot \sin(x)}{\cos(x)}|=|x\cdot \tan(x)|\label{eqn:pert_cos}
\end{align}
From Proposition~\ref{prop:cos_condition} we conclude that $|x\cdot \tan(x)|<\tan(1)$ for $x\in (-1,1)$. Thus, we can approximate $\cos(x)$ in the base interval in one pass.
 
\subsection{Extending Base Interval Approximations}\label{subsec:extend}
In Section~\ref{subsubsec:base_sin} and \ref{subsubsec:base_cos}, we discussed approximations for $\sin(x)$ and $\cos(x)$ in the base interval. In what follows, we show that range reduction identities can be used to extend these approximations to calculate sine and cosine for an argument $x$ represented by $(m,n,p)$ where $|\frac{m}{n}|\geq 1$.

In our calculations, we use a representation $(m',n',p)$ of $\pi$. This representation can be obtained based on our approximation for $\arctan(x)$ (see Section~\ref{subsec:arctan}) and the following identity:
\begin{align*}
\pi=4\arctan(1)
\end{align*}

\subsubsection{Sine}\label{subsubsec:extend_sin}

\begin{mytheorem}\label{thm:extend_sin}
Let $x$ be a real number represented by $(m,n,p)$ such that $|\frac{m}{n}|\geq 1$ and $(m',n',p)$ be a representation for $\pi$. Suppose $\frac{\om}{\on}=\frac{m}{n}+k\frac{m'}{n'}$ where $k\in\integer$ and $0<\frac{\om}{\on}<\frac{m'}{n'}$. The value of $\sin(x)$ can be approximated as follows:
\begin{enumerate}[i.]
\item \label{thm:extend_sin_1}
If $0<\frac{\om}{\on}<1$, then $\sin(x)$ can be represented by $(m_1,n_1,p-i_1)$ where:
\begin{align*}
i_1=&s_1+t_1\\
\frac{m_1}{n_1}=&\taylorsin{0}{N_1}{(\frac{\om}{\on})}
\end{align*}
\item \label{thm:extend_sin_2}
If $1\leq \frac{\om}{\on}<2$, then $\sin(x)$ can be represented by $(m_2,n_2,p-i_2)$ where:
\begin{align*}
i_2=&s_1+t_2+2\\
\frac{m_2}{n_2}=&2\bigg(\taylorsin{0}{N_1}{(\frac{\om}{2\on})}\bigg)\cdot\bigg(\taylorcos{0}{N_2}{(\frac{\om}{2\on})}\bigg)
\end{align*}
\item \label{thm:extend_sin_3}
If $2\leq \frac{\om}{\on}<\frac{m'}{n'}$, then $\sin(x)$ can be represented by $(m_3,n_3,p-i_3)$ where:
\begin{align*}
i_3=&s_1+s_2+t_3+6\\
\frac{m_3}{n_3}=&8\bigg(\taylorsin{0}{N_1}{(\frac{\om}{4\on})}\bigg)\cdot\bigg(\taylorcos{0}{N_2}{(\frac{\om}{4\on})}\bigg)\\
& \cdot\bigg(\taylorsin{0}{N_3}{(\frac{m'}{4n'}-\frac{\om}{4\on})}\bigg)\cdot\bigg(\taylorcos{0}{N_4}{(\frac{m'}{4n'}-\frac{\om}{4\on})}\bigg)
\end{align*}
\end{enumerate}
In the approximations above, $s_1,s_2,N_1,N_2,N_3,N_4\in \nat$ are the smallest natural numbers satisfying:
\begin{gather*}
s_1\geq \log_2(\frac{1+\frac{\min(|\frac{m}{n}|,k\frac{m'}{n'})}{\max(|\frac{m}{n}|,k\frac{m'}{n'})}}{1-\frac{\min(|\frac{m}{n}|,k\frac{m'}{n'})}{\max(|\frac{m}{n}|,k\frac{m'}{n'})}}) ~,~s_2\geq \log_2(\frac{1+\frac{\min(\frac{m'}{4n'},\frac{\om}{4\on})}{\max(\frac{m'}{4n'},\frac{\om}{4\on})}}{1-\frac{\min(\frac{m'}{4n'},\frac{\om}{4\on})}{\max(\frac{m'}{4n'},\frac{\om}{4\on})}})\\
(\frac{5}{6})(\frac{(2N_1+2)!(2N_1-1)^2}{2^{2N_1}})>2^{p-s_1} ,\frac{(2N_2+1)!(2N_2-2)^2}{2^{2N_2}}>2^{p-s_1}\\
(\frac{5}{6})(\frac{(2N_3+2)!(2N_3-1)^2}{2^{2N_3}})>2^{p-s_1-s_2},\frac{(2N_4+1)!(2N_4-2)^2}{2^{2N_4}}>2^{p-s_1-s_2}
\end{gather*}
and $t_1,t_2\in \nat$ are defined as follows:
\begin{align*}
t_1=&2\lceil\log_2(2N_1-1) \rceil+3\\
t_2=&\max(2\lceil\log_2(2N_1-1) \rceil+3,2\lceil\log_2(2N_2-2) \rceil+3)\\
t_3=&\max(2\lceil\log_2(2N_1-1) \rceil+3,2\lceil\log_2(2N_2-2) \rceil+3,\\
&\phantom{\max(} 2\lceil\log_2(2N_3-1) \rceil+3,2\lceil\log_2(2N_4-2) \rceil+3)
\end{align*}

\end{mytheorem}

\begin{proof}
Since $|\frac{m}{n}|\geq 1$, we can choose $k\in \integer$ such that $\frac{\om}{\on}=\frac{m}{n}+k\frac{m'}{n'}$ and $0< \frac{\om}{\on}<\frac{m'}{n'}$. Suppose $y=x+k\pi$. We use the following identity to calculate $\sin(x)$:
\begin{align}
\sin(x)=\sin(x+k\pi)=\sin(y)\label{eqn:red_sin}
\end{align}
We approximate $y$ by $\frac{\om}{\on}=\frac{m}{n}+k\frac{m'}{n'}$ (see Theorem~\ref{thm:add}.\ref{thm:add_diff_sign}). By performing this approximation, we lose $s_1\in\nat$ units of precision where $s_1$ is the smallest number satisfying:
\begin{align*}
s_1\geq \log_2(\frac{1+\frac{\min(|\frac{m}{n}|,k\frac{m'}{n'})}{\max(|\frac{m}{n}|,k\frac{m'}{n'})}}{1-\frac{\min(|\frac{m}{n}|,k\frac{m'}{n'})}{\max(|\frac{m}{n}|,k\frac{m'}{n'})}})
\end{align*}
We consider three cases for calculating $\sin(y)$:
\begin{enumerate}
\item
Suppose $0<\frac{\om}{\on}<1$. We apply Theorem~\ref{thm:base_sin}. Thus, $t_1=2\lceil\log_2(2N_1-1) \rceil+3$ units of precision is lost in the approximation of $\sin(y)$. A total of $s_1+t_1$ units of precision is lost in the approximation of $\sin(x)$.
\item
Suppose $1\leq \frac{\om}{\on}<2$. We use the following identity to bring the argument within the base interval:
\begin{align}
\sin(y)&=2\sin(\frac{y}{2})\cos(\frac{y}{2})\label{eqn:sin_id1}
\end{align}
We approximate $\frac{y}{2}$ by $\frac{\om}{2\on}$. Since $\frac{1}{2}\leq \frac{\om}{2\on}<1$, we apply Theorem~\ref{thm:base_sin} and \ref{thm:base_cos} to approximate $\sin(\frac{y}{2})$ and $\cos(\frac{y}{2})$, respectively. Loss of precision in these calculations is as follows:
\begin{align*}
t_2=\max(2\lceil\log_2(2N_1-1) \rceil+3,2\lceil\log_2(2N_2-2) \rceil+3)
\end{align*}
The multiplication in equality~\eqref{eqn:sin_id1} reduces the precision by $2$ units (see Theorem~\ref{thm:mult}). A total of $s_1+t_2+2$ units of precision is lost in the approximation of $\sin(x)$.
\item
Suppose $2\leq \frac{\om}{\on}<\frac{m'}{n'}$. We use the following identity to bring the argument within the base interval. 
\begin{align}
\sin(y)&=2\sin(\frac{y}{2})\cos(\frac{y}{2})\nonumber\\
&=4\sin(\frac{y}{4})\cos(\frac{y}{4})\cos(\frac{y}{2})\nonumber\\
&=4\sin(\frac{y}{4})\cos(\frac{y}{4})\sin(\frac{\pi}{2}-\frac{y}{2})\nonumber\\
&=8\sin(\frac{y}{4})\cos(\frac{y}{4})\sin(\frac{\pi}{4}-\frac{y}{4})\cos(\frac{\pi}{4}-\frac{y}{4})\label{eqn:sin_id2}
\end{align}
We approximate $\frac{y}{4}$ and $\frac{\pi}{4}-\frac{y}{4}$ by $\frac{\om}{4\on}$ and $\frac{m'}{4n'}-\frac{\om}{4\on}$, respectively. We lose $s_2\in \nat$ units of precision in this approximation (see Theorem~\ref{thm:add}.\ref{thm:add_diff_sign}) where $s_2$ is the smallest number satisfying:
\begin{align*}
s_2\geq \log_2(\frac{1+\frac{\min(\frac{m'}{4n'},\frac{\om}{4\on})}{\max(\frac{m'}{4n'},\frac{\om}{4\on})}}{1-\frac{\min(\frac{m'}{4n'},\frac{\om}{4\on})}{\max(\frac{m'}{4n'},\frac{\om}{4\on})}})
\end{align*}
We apply Theorem~\ref{thm:base_sin} and \ref{thm:base_cos} to approximate $\sin(\frac{y}{4}), \cos(\frac{y}{4}), \sin(\frac{\pi}{4}-\frac{y}{4})$, and $\cos(\frac{\pi}{4}-\frac{y}{4})$. Lose of precision in these approximations can be calculated as follows:
\begin{align*}
t_3=\max(&2\lceil\log_2(2N_1-1) \rceil+3,2\lceil\log_2(2N_2-2) \rceil+3,\\
&2\lceil\log_2(2N_3-1) \rceil+3,2\lceil\log_2(2N_4-2) \rceil+3)
\end{align*}
The three multiplications in $\sin(\frac{y}{4})\cos(\frac{y}{4})\sin(\frac{\pi}{4}-\frac{y}{4})\cos(\frac{\pi}{4}-\frac{y}{4})$ reduce the precision by $6$ units (see Theorem~\ref{thm:mult}). A total of $s_1+s_2+t_3+6$ units of precision is lost in the approximation of $\sin(x)$.
\end{enumerate}
\end{proof}

\begin{algorithm}
\caption{Sine}\label{alg:sin}
\begin{algorithmic}[1]
\Require $\expr$ has the shape $\sin{x}$
\Procedure{Compute}{$\expr,p$}
\State Choose an odd $N$ such that $(\frac{5}{6})(\frac{(2N+2)!(N-1)^2}{2^{2N}})>2^{p+2\lceil\log_2(2N-1) \rceil+3}$
\State $p_x\gets p+2\lceil\log_2(2N-1) \rceil+3$
\Repeat \label{algline:sin_go}
\State $\frac{m}{n}\gets \Call{Compute}{x,p_x}$
\If{$-1<\frac{m}{n}<1$} \Comment{Theorem~\ref{thm:base_sin}}
	\State $\frac{m_0}{n_0}\gets \taylorsin{0}{N}{(\frac{m}{n})}$
	\State \Return $\frac{m_0}{n_0}$
\Else
	\State $\frac{m'}{n'}\gets \Call{Compute}{4\arctan(1),p_x}$
	\State Choose $k\in\integer$ such that $0<\frac{m}{n}+k\frac{m'}{n'}<\frac{m'}{n'}$
	\State $\frac{\om}{\on}\gets \frac{m}{n}+k\frac{m'}{n'}$
	\State Choose $s_1\in\nat$ such that $s_1\geq \log_2(\frac{1+\frac{\min(|\frac{m}{n}|,k\frac{m'}{n'})}{\max(|\frac{m}{n}|,k\frac{m'}{n'})}}{1-\frac{\min(|\frac{m}{n}|,k\frac{m'}{n'})}{\max(|\frac{m}{n}|,k\frac{m'}{n'})}})$
	\State Choose an odd $N_1$ such that $(\frac{5}{6})(\frac{(2N_1+2)!(2N_1-1)^2}{2^{2N_1}})>2^{p_x-s_1}$
	\State Choose an odd $N_2$ such that $\frac{(2N_2+1)!(2N_2-2)^2}{2^{2N_2}}>2^{p_x-s_1}$
	\If{$0<\frac{\om}{\on}<1$} \Comment{Theorem~\ref{thm:extend_sin}.\ref{thm:extend_sin_1}}
		\State $t_1\gets 2\lceil\log_2(2N_1-1) \rceil+3$
		\If{$p_x-s_1-t_1\geq p$}
			\State $\frac{m_1}{n_1}=\taylorsin{0}{N_1}{(\frac{\om}{\on})}$
			\State \Return $\frac{m_1}{n_1}$
		\Else
			\State $p_x\gets p_x+1$ \label{algline:sin_inc1}
		\EndIf	
	\ElsIf{$1\leq \frac{\om}{\on}<2$} \Comment{Theorem~\ref{thm:extend_sin}.\ref{thm:extend_sin_2}}
		\State $t_2\gets \max(2\lceil\log_2(2N_1-1) \rceil+3,2\lceil\log_2(2N_2-2) \rceil+3)$
		\If{$p_x-s_1-t_2-2\geq p$}
			\State $\frac{m_2}{n_2}=2\big(\taylorsin{0}{N_1}{(\frac{\om}{2\on})}\big)\big(\taylorcos{0}{N_2}{(\frac{\om}{2\on})}\big)$
			\State \Return $\frac{m_2}{n_2}$
		\Else
			\State $p_x\gets p_x+1$	\label{algline:sin_inc2}
		\EndIf
	\Else	\Comment{Theorem~\ref{thm:extend_sin}.\ref{thm:extend_sin_3}}
		\State \hspace{-1mm} Choose an odd $N_3$ such that \hspace{-1mm} $(\frac{5}{6})(\frac{(2N_3+2)!(2N_3-1)^2}{2^{2N_3}})>2^{p_x-s_1-s_2}$
		\State Choose an odd $N_4$ such that $\frac{(2N_4+1)!(2N_4-2)^2}{2^{2N_4}}>2^{p_x-s_1-s_2}$
		\State \begin{varwidth}[t]{\linewidth}$t_3\gets \max(2\lceil\log_2(2N_1-1) \rceil+3, 2\lceil\log_2(2N_2-2) \rceil+3,$\par
		\hskip\algorithmicindent $\phantom{~~\max (}2\lceil\log_2(2N_3-1) \rceil+3, 2\lceil\log_2(2N_4-2) \rceil+3)$
		\end{varwidth}
		\algstore{sinalg}
\end{algorithmic}
\end{algorithm}

\begin{algorithm}
\ContinuedFloat
\caption{Sine (Continued)}
\begin{algorithmic}[1]
\algrestore{sinalg}
		\If{$p_x-s_1-s_2-t_3-6\geq p$}
			\State $\frac{m_3}{n_3}=8\big(\taylorsin{0}{N_1}{(\frac{\om}{4\on})}\big)\big(\taylorcos{0}{N_2}{(\frac{\om}{4\on})}\big)$\par
		\hskip\algorithmicindent $~~~~~~~~\phantom{\frac{m_3}{n_3}=8}\big(\taylorsin{0}{N_3}{(\frac{m'}{4n'}- \frac{\om}{4\on})}\big)\big(\taylorcos{0}{N_4}{(\frac{m'}{4n'}-\frac{\om}{4\on})}\big)$
			\State \Return $\frac{m_3}{n_3}$
		\Else
			\State $p_x\gets p_x+1$ \label{algline:sin_inc3}
		\EndIf
	\EndIf
\EndIf
\Until{$\true$}
\EndProcedure
\end{algorithmic}
\end{algorithm}

Algorithm~\ref{alg:sin} applies Theorem~\ref{thm:base_sin} and \ref{thm:extend_sin} to approximate $\sin(x)$ with arbitrary precision. In this algorithm, initially, we calculate $x$ with a precision that is adequate for calculations in the base interval. However, if the obtained approximation is outside the base interval, we use the half-angle formula or add rational multiples of $\pi$ to the argument (see equality~\eqref{eqn:red_sin} and \eqref{eqn:sin_id2}). 

Observe that an arbitrary amount of precision can be lost in the approximation of $x+k\pi$ and $\frac{\pi}{4}-\frac{y}{4}$ when $x\approx -k\pi$. Hence, iterative computations might be necessary in the approximation (see Line~\ref{algline:sin_go},\ref{algline:sin_inc1},\ref{algline:sin_inc2},\ref{algline:sin_inc3} in Algorithm~\ref{alg:sin}). To show that this is essential for sine, we reconsider the perturbation analysis in equality~\eqref{eqn:pert_sin} for $f(x)=\sin(x)$. The quantity $|x\cdot\cot(x)|$ can be arbitrary large for $x\approx -k\pi$. Thus, iterative computations are essential for approximating $\sin(x)$.

\subsubsection{Cosine}\label{subsubsec:extend_cos}
\begin{mytheorem}\label{thm:extend_cos}
Let $x$ be a real number represented by $(m,n,p)$ such that $|\frac{m}{n}|\geq 1$ and $(m',n',p)$ be a representation for $\pi$. Suppose $\frac{\om}{\on}=\frac{m}{n}+(2k+1)\frac{m'}{2n'}$ where $k\in\integer$ and $0<\frac{\om}{\on}<\frac{m'}{n'}$. The value of $\cos(x)$ can be approximated as follows:
\begin{enumerate}[i.]
\item \label{thm:extend_cos_1}
If $0<\frac{\om}{\on}<1$, then $\cos(x)$ can be represented by $(m_1,n_1,p-i_1)$ where:
\begin{align*}
i_1=&s_1+t_1\\
\frac{m_1}{n_1}=&(-1)^k\taylorsin{0}{N_1}{(\frac{\om}{\on})}
\end{align*}
\item \label{thm:extend_cos_2}
If $1\leq \frac{\om}{\on}<2$, then $\cos(x)$ can be represented by $(m_2,n_2,p-i_2)$ where:
\begin{align*}
i_2=&s_1+t_2+2\\
\frac{m_2}{n_2}=&2\cdot(-1)^k\bigg( \taylorsin{0}{N_1}{(\frac{\om}{2\on})}\bigg)\cdot\bigg(\taylorcos{0}{N_2}{(\frac{\om}{2\on})}\bigg)
\end{align*}
\item \label{thm:extend_cos_3}
If $2\leq \frac{\om}{\on}<\frac{m'}{n'}$, then $\cos(x)$ can be represented by $(m_3,n_3,p-i_3)$ where:
\begin{align*}
i_3=&s_1+s_2+t_3+6\\
\frac{m_3}{n_3}=&8\cdot (-1)^k \bigg(\taylorsin{0}{N_1}{(\frac{\om}{4\on})}\bigg)\cdot\bigg(\taylorcos{0}{N_2}{(\frac{\om}{4\on})}\bigg)\\
& \cdot\bigg(\taylorsin{0}{N_3}{(\frac{m'}{4n'}-\frac{\om}{4\on})}\bigg)\cdot\bigg(\taylorcos{0}{N_4}{(\frac{m'}{4n'}-\frac{\om}{4\on})}\bigg)
\end{align*}
\end{enumerate}
In the approximations above, $s_1,s_2,N_1,N_2,N_3,N_4\in\nat$ are the smallest natural numbers satisfying:
\begin{align*}
s_1\geq \log_2(\frac{1+\frac{\min(|\frac{m}{n}|,(2k+1)\frac{m'}{2n'})}{\max(|\frac{m}{n}|,(2k+1)\frac{m'}{2n'})}}{1-\frac{\min(|\frac{m}{n}|,(2k+1)\frac{m'}{2n'})}{\max(|\frac{m}{n}|,(2k+1)\frac{m'}{2n'})}})~&,~s_2\geq \log_2(\frac{1+\frac{\min(\frac{m'}{4n'},\frac{\om}{4\on})}{\max(\frac{m'}{4n'},\frac{\om}{4\on})}}{1-\frac{\min(\frac{m'}{4n'},\frac{\om}{4\on})}{\max(\frac{m'}{4n'},\frac{\om}{4\on})}})\\
(\frac{5}{6})(\frac{(2N_1+2)!(2N_1-1)^2}{2^{2N_1}})>2^{p-s_1} ~&,~\frac{(2N_2+1)!(2N_2-2)^2}{2^{2N_2}}>2^{p-s_1}\\
(\frac{5}{6})(\frac{(2N_3+2)!(2N_3-1)^2}{2^{2N_3}})>2^{p-s_1-s_2}~&,~\frac{(2N_4+1)!(2N_4-2)^2}{2^{2N_4}}>2^{p-s_1-s_2}
\end{align*}
and $t_1,t_2\in \nat$ are defined as follows:
\begin{align*}
t_1=&2\lceil\log_2(2N_1-1) \rceil+3\\
t_2=&\max(2\lceil\log_2(2N_1-1) \rceil+3,2\lceil\log_2(2N_2-2) \rceil+3)\\
t_3=&\max(2\lceil\log_2(2N_1-1) \rceil+3,2\lceil\log_2(2N_2-2) \rceil+3,\\
&\phantom{\max(} 2\lceil\log_2(2N_3-1) \rceil+3,2\lceil\log_2(2N_4-2) \rceil+3)
\end{align*}
\end{mytheorem}

\begin{proof}
Since $|\frac{m}{n}|\geq 1$, we can choose $k\in \integer$ such that $\frac{\om}{\on}=\frac{m}{n}+(2k+1)\frac{m'}{2n'}$ and $0< \frac{\om}{\on}<\frac{m'}{n'}$. Suppose $y=x+(2k+1)\frac{\pi}{2}$. We use the following identity to calculate $\cos(x)$:
\begin{align*}
\cos(x)=(-1)^k\sin(x+(2k+1)\frac{\pi}{2})=(-1)^k\sin(y)
\end{align*}
We approximate $y$ by $\frac{\om}{\on}=\frac{m}{n}+(2k+1)\frac{m'}{2n'}$. From Theorem~\ref{thm:add}.\ref{thm:add_diff_sign}, we can conclude that $s_1\in \nat$ units of precision will be lost in this approximation where $s_1$ is the smallest natural number satisfying:
\begin{align*}
s_1\geq \log_2(\frac{1+\frac{\min(|\frac{m}{n}|,(2k+1)\frac{m'}{2n'})}{\max(|\frac{m}{n}|,(2k+1)\frac{m'}{2n'})}}{1-\frac{\min(|\frac{m}{n}|,(2k+1)\frac{m'}{2n'})}{\max(|\frac{m}{n}|,(2k+1)\frac{m'}{2n'})}})
\end{align*}
We apply Theorem~\ref{thm:extend_sin} to approximate $\sin(y)$ and determine the loss of precision in the approximation. From Theorem~\ref{thm:unary_neg} we can conclude that the factor $(-1)^k$ does not influence the precision of the approximation. 
\end{proof}

\begin{algorithm}
\caption{Cosine}\label{alg:cos}
\begin{algorithmic}[1]
\Require $\expr$ has the shape $\cos{x}$
\Procedure{Compute}{$\expr,p$}
\State Choose an odd $N$ such that $\frac{(2N+1)!(2N-2)^2}{2^{2N}}>2^{p+2\lceil\log_2(2N-2) \rceil +2}$
\State $p_x\gets p+2\lceil\log_2(2N-2) \rceil+2$
\Repeat \label{algline:cos_go}
\State $\frac{m}{n}\gets \Call{Compute}{x,p_x}$
\If{$-1<\frac{m}{n}<1$} \Comment{Theorem~\ref{thm:base_cos}}
	\State $\frac{m_0}{n_0}=\taylorcos{0}{N}{(\frac{m}{n})}$
	\State \Return $\frac{m_0}{n_0}$
\Else
	\State $\frac{m'}{n'}\gets \Call{Compute}{4\arctan(1),p_x}$
	\State Choose $k\in\integer$ such that $0<\frac{m}{n}+(2k+1)\frac{m'}{2n'}<\frac{m'}{n'}$
	\State $\frac{\om}{\on}\gets \frac{m}{n}+(2k+1)\frac{m'}{2n'}$
	\State Choose $s_1\in\nat$ such that $s_1\geq \log_2(\frac{1+\frac{\min(|\frac{m}{n}|,(2k+1)\frac{m'}{2n'})}{\max(|\frac{m}{n}|,(2k+1)\frac{m'}{2n'})}}{1-\frac{\min(|\frac{m}{n}|,(2k+1)\frac{m'}{2n'})}{\max(|\frac{m}{n}|,(2k+1)\frac{m'}{2n'})}})$
	\State Choose an odd $N_1$ such that $(\frac{5}{6})(\frac{(2N_1+2)!(2N_1-1)^2}{2^{2N_1}})>2^{p_x-s_1}$
	\State Choose an odd $N_2$ such that $\frac{(2N_2+1)!(2N_2-2)^2}{2^{2N_2}}>2^{p_x-s_1}$
	\If{$0<\frac{\om}{\on}<1$} \Comment{Theorem~\ref{thm:extend_cos}.\ref{thm:extend_cos_1}}
		\State $t_1\gets 2\lceil\log_2(2N_1-1) \rceil +3$
		\If{$p_x-s_1-t_1\geq p$}
			\State $\frac{m_1}{n_1}=(-1)^k\taylorsin{0}{N_1}{(\frac{\om}{\on})}$
			\State \Return $\frac{m_1}{n_1}$
		\Else
			\State $p_x\gets p_x +1$ \label{algline:cos_inc1}
		\EndIf
	\ElsIf{$1\leq \frac{\om}{\on}<2$} \Comment{Theorem~\ref{thm:extend_cos}.\ref{thm:extend_cos_2}}
		\State $t_2\gets \max(2\lceil\log_2(2N_1-1) \rceil+3,2\lceil\log_2(2N_2-2) \rceil+3)$
		\If{$p_x-s_1-t_2-2\geq p$}
			\State $\frac{m_2}{n_2}=2\cdot (-1)^k\big(\taylorsin{0}{N_1}{(\frac{\om}{2\on})} \big)\big(\taylorcos{0}{N_2}{(\frac{\om}{2\on})}\big)$
			\State \Return $\frac{m_2}{n_2}$
		\Else
			\State $p_x\gets p_x +1$ \label{algline:cos_inc2}
		\EndIf	
	\Else \Comment{Theorem~\ref{thm:extend_cos}.\ref{thm:extend_cos_3}}
		\State Choose an odd $N_3$ such that $(\frac{5}{6})(\frac{(2N_3+2)!(2N_3-1)^2}{2^{2N_3}})>2^{p-s_1-s_2}$
		\State Choose an odd $N_4$ such that $\frac{(2N_4+1)!(2N_4-2)^2}{2^{2N_4}}>2^{p-s_1-s_2}$
		\State \begin{varwidth}[t]{\linewidth} $t_3\gets \max(2\lceil\log_2(2N_1-1) \rceil+3,2\lceil\log_2(2N_2-2) \rceil+3,$\par
		 \hskip\algorithmicindent$\phantom{~~\max (} 2\lceil\log_2(2N_3-1) \rceil+3,2\lceil\log_2(2N_4-2) \rceil+3)$
		 \end{varwidth}
		 \algstore{cosalg}
\end{algorithmic}
\end{algorithm}

\begin{algorithm}
\ContinuedFloat
\caption{Cosine (Continued)}
\begin{algorithmic}[1]
\algrestore{cosalg}
		\If{$p_x-s_1-s_2-t_3\geq p$}
			\State \begin{varwidth}[t]{\linewidth}$\frac{m_3}{n_3}=8\cdot (-1)^k \big(\taylorsin{0}{N_1}{(\frac{\om}{4\on})}\big)\big(\taylorcos{0}{N_2}{(\frac{\om}{4\on})}\big)$
			\hskip\algorithmicindent$\phantom{\frac{m_3}{n_3}=}\big(\taylorsin{0}{N_3}{(\frac{m'}{4n'}-\frac{\om}{4\on})}\big)\big(\taylorcos{0}{N_4}{(\frac{m'}{4n'}-\frac{\om}{4\on})}\big)$
			\end{varwidth}
			\State \Return $\frac{m_3}{n_3}$
		\Else
			\State $p_x\gets p_x +1$ \label{algline:cos_inc3}
		\EndIf
	\EndIf
\EndIf
\Until{$\true$}
\EndProcedure
\end{algorithmic}
\end{algorithm}

Algorithm~\ref{alg:cos} applies Theorem~\ref{thm:base_cos} and \ref{thm:extend_cos} to approximate $\cos(x)$ with a desired precision. The algorithm first calculates $x$ with a precision that is sufficient for approximating $\cos(x)$ in the base interval. Similar to $\sin(x)$, we use range reduction identities to calculate $\cos(x)$ for an arbitrary $x$. 

Observe that a significant amount of precision might be lost in the approximation when $x\approx \frac{-(2k+1)}{2}\pi$. Thus, iterative computations might be necessary in our approximation (see Line~\ref{algline:cos_go},\ref{algline:cos_inc1},\ref{algline:cos_inc2},\ref{algline:cos_inc3} in Algorithm~\ref{alg:cos}). To show that these recomputations are essential, we reconsider the perturbation analysis of equality~\eqref{eqn:pert_cos}. The quantity $|x\cdot \tan(x)|$ can be arbitrary large for $x\approx \frac{-(2k+1)}{2}\pi$ and hence iterative computations are unavoidable for $\cos(x)$.

%% file: related_work.tex
\section{Related Work}\label{sec:related_work}
An implementation for a bottom-up approach to exact real arithmetic is proposed in \cite{M01}. For a given expression, the inputs are computed with predefined precisions and a bottom-up scheme is used to determine the guaranteed precision of the output. Iterative computations are required if the obtained precision is not adequate. A formalization of a top-down approach in a theorem prover is proposed in \cite{O08}. The author first provides a definition for a metric space based on a ball relation. Afterwards, real numbers are defined as the completion of the metric space $\mathbb{Q}$. Rational operations are lifted to approximate operations on real numbers. This approach is optimized in \cite{KS11}.

Two closely-related top-down approaches based on absolute errors have been studied in \cite{GL00,M05}. These approaches mainly differ in their approximations of the transcendental functions. In \cite{GL00} the authors introduce a general way for calculating with Taylor expansions and apply this method to approximate the transcendental functions. In \cite{M05}, the approximations of the transcendental functions are treated separately and in a more ad-hoc way. In contrast, our approach is based on relative errors. We provide detailed proofs of correctness for each operation and use perturbation analysis to identify essential recomputations.

In \cite{CNR11} the authors propose a layered framework for  computations with real numbers. The lowest layer is an implementation of floating point arithmetic. In the second layer, arithmetic operations are approximated using polynomial models. The highest layer supports more advanced features such as differential operations. Proof of correctness in Coq and an implementation based on \cite{CNR11} are also available. In this article, we have focused on approximating specific operations, whereas in \cite{CNR11} polynomial models are discussed in an abstract way without concrete examples from well-known arithmetic operations. 

The approach introduced in \cite{P98} is based on linear fractional transformations (LFTs). Computations are encoded as trees of LFTs; various operations are defined to extract the result of a computation from the corresponding tree. However, this approach does not specify a top-down scheme to relate the desired precision in the output and the required operations on the expression tree. Expression trees are evaluated using lazy evaluation; computations terminate when adequate information is available in the root of a tree. 

A symbolic approach to exact real arithmetic has been proposed in \cite{K00}. The author uses infinite binary sequences in the golden ratio base to represent real numbers. To calculate an expression, first the symbolic techniques available in Maple are applied to obtain a simplified expression. Additional Maple procedures are implemented by the author to extract binary sequences from simplified expressions. Performing operations on binary sequences is also possible. However, choosing a suitable balance between symbolic computations and direct manipulation of binary sequences depends on the given expression. As indicated in \cite{K00}, using this approach to its full potentials requires expertise in Maple. Moreover, the procedure might need adaptations for each problem.

%% file: conclusion.tex
\section{Conclusion}
\label{sec:conclusion}
In this article, we proposed a simple representation for real numbers and discussed a top-down approach for approximating various arithmetic operations with arbitrary precision. The focus was on:
\begin{itemize}
\item
providing complete algorithms and proofs of correctness for the approximations, and
\item
perturbation analysis to identify essential iterative computations.
\end{itemize}

Existing exact real arithmetic approaches have explored different representations for real numbers; approximations for algebraic operations and transcendental functions have also been proposed based on these representations. As far as we can see, proofs of correctness for existing approaches are restricted to basic operations. Moreover, no formal reasoning is provided to prove the necessity of iterative computations.

We envisage various extensions of the presented approach. From a practical point of view, some optimizations are essential. For example, the coefficients $m$ and $n$ in the representation $(m,n,p)$ can grow rapidly during computations. Thus, space efficiency is a relevant concern. One can consider an alternative representation in which large coefficients are represented in a more efficient way. Moreover, the computational efficiency of the transcendental functions can be improved by reducing the amount of required computations (i.e, number of rectangles in Riemann sums, number of terms in Taylor expansions) to guarantee the desired precision. 

As discussed in Section~\ref{subsec:add}, in certain computational problems, computing the expressions as they are would lead to loss of precision, whereas rewriting the expressions would allow us to compute them in one pass. Our top-down approach can be extended with a set of rewrite rules that transform problematic expressions into expressions that can be calculated in one pass. 

%% file: useful_facts.tex
\section{Useful Propositions \& Lemmas}\label{app:prop_lem}
\begin{proposition}\label{prop:prec2}
For any $p\in \nat$ the following inequalities hold:
\begin{align}
(1+\frac{1}{2^{p}})^2\leq 1+\frac{1}{2^{p-2}} \label{eqn:posprec2}\\
1-\frac{1}{2^{p-2}}\leq (1-\frac{1}{2^{p}})^2 \label{eqn:negprec2}
\end{align}
\end{proposition}
\begin{proof}
For inequality~\eqref{eqn:posprec2} we can write:
\begin{align*}
(1+\frac{1}{2^{p}})^2=1+\frac{1}{2^{2p}}+\frac{1}{2^{p-1}}\leq 1+\frac{1}{2^{p-1}}+\frac{1}{2^{p-1}}=1+\frac{1}{2^{p-2}}
\end{align*}
Similarly for inequality~\eqref{eqn:negprec2} we have:
\begin{align*}
(1-\frac{1}{2^{p}})^2=1+\frac{1}{2^{2p}}-\frac{1}{2^{p-1}}\geq 1-\frac{1}{2^{p-1}}>1-\frac{1}{2^{p-2}}
\end{align*}

\end{proof}

\begin{proposition}\label{prop:prec1}
For any $p\in \nat^{+}$ the following inequalities hold:
\begin{align}
\frac{2^{p}}{2^{p}-1}\leq 1+\frac{1}{2^{p-1}}\label{eqn:posprec1}\\
1-\frac{1}{2^{p-1}}\leq \frac{2^{p}}{2^{p}+1}\label{eqn:negprec1}
\end{align}
\end{proposition}

\begin{proof}
Since $p\in \nat^+$ we have $2^p\geq 2$. For inequality~\eqref{eqn:posprec1} we have:
\begin{align*}
\frac{2^{p}}{2^{p}-1}=\frac{2^{p}-1+1}{2^{p}-1}=1+\frac{1}{2^{p}-1}\leq 1+\frac{1}{2^{p-1}}
\end{align*}
Similarly for inequality~\eqref{eqn:negprec1} we can write:
\begin{align*}
\frac{2^{p}}{2^{p}+1}=\frac{2^{p}+1-1}{2^{p}+1}=1-\frac{1}{2^{p}+1}\geq 1-\frac{1}{2^{p-1}}
\end{align*}

\end{proof}

\begin{proposition}\label{prop:ln_bounds}
For any $0<y<1$ the following inequality holds:
\begin{align}
\log_2(1+y)\leq \frac{y}{\ln(2)}
\end{align}
\end{proposition}

\begin{proof}
Based on Taylor's theorem, we can write the following expansion for $\log_2(1+y)$:
\begin{align*}
\log_2(1+y)=\sum_{i=1}^{N}(-1)^{i-1}\frac{y^i}{i\ln(2)}+\frac{(-1)^Ny^{N+1}}{(N+1)(1+c_y)^{N+1}\ln(2)}
\end{align*}
where $0<c_y<y$. For $N=1$ we obtain:
\begin{align*}
\log_2(1+y)=\frac{y}{\ln(2)}-\frac{y^2}{2(1+c_1)^2\ln(2)}
\end{align*}
The quantity $\frac{y^2}{2(1+c_1)^2\ln(2)}$ is positive and hence $\log_2(1+y)\leq \frac{y}{\ln(2)}$. 

\end{proof}

\begin{proposition}\label{prop:arctan_condition}
For $x\in \real\setminus\{0\}$ we have $|\frac{x}{(1+x^2)\arctan(x)}|<1$.  
\end{proposition}

\begin{proof}
We analyze the derivative of $f(x)=|\frac{x}{(1+x^2)\arctan(x)}|=\frac{x}{(1+x^2)\arctan(x)}$ and find the intervals in which $f(x)$ is increasing/decreasing. 
\begin{align*}
f'(x)=\frac{\arctan(x)(1-x^2)-x}{(1+x^2)^2(\arctan (x))^2}
\end{align*}
We analyze the derivative in the following cases:
\begin{enumerate}
\item
Suppose $0<x<1$. We use the Taylor expansion of $\arctan(x)$ to rewrite $\arctan(x)(1-x^2)-x$:
\begin{align*}
\arctan(x)(1-x^2)-x&=\taylorarctan{0}{\infty}{x}-x-x^2\arctan(x)\\
&=\taylorarctan{1}{\infty}{x}-x^2\arctan(x)\\
&=\sum_{i=1}^{\infty}{-x^{4i-1}(\frac{1}{4i-1}-\frac{x^2}{4i+1})}-x^2\arctan(x)
\end{align*}
The term $-x^2\arctan(x)$ is negative and $-x^{4i-1}(\frac{1}{4i-1}-\frac{x^2}{4i+1})$ is negative for $i\geq 1$ and $0<x<1$. Thus, the summation $\sum_{i=1}^{\infty}{-x^{4i-1}(\frac{1}{4i-1}-\frac{x^2}{4i+1})}-x^2\arctan(x)$ is negative and $f'(x)<0$. The function $f(x)$ is decreasing for $0<x<1$.
\item
Suppose $x\geq 1$. In this case, $\arctan(x)(1-x^2)-x<0$. Thus, $f'(x)<0$ and $f(x)$ is decreasing for $x\geq 1$.
\item
Suppose $x<0$. Since $f'(-x)=-f'(x)$, we can conclude from the first two cases that $f(x)$ is increasing when $x<0$.
\end{enumerate}
The case analysis shows that $f(x)$ is bounded from above by $\lim_{x\to 0}f(x)=1$.

\end{proof}

\begin{proposition}\label{prop:sin_sign}
For any $y\in (-1,0) \cup (0,1)$ and $i\in \nat$ the following inequalities hold:
\begin{align*}
\frac{(-1)^{2i}}{(4i+1)!}y^{4i+1}+\frac{(-1)^{2i+1}}{(4i+3)!}y^{4i+3}>0 \qquad \text{if }y>0 \\
\frac{(-1)^{2i}}{(4i+1)!}y^{4i+1}+\frac{(-1)^{2i+1}}{(4i+3)!}y^{4i+3}<0 \qquad \text{if }y<0
\end{align*}
\end{proposition}

\begin{proof}
We can rewrite the summation as follows:
\begin{align*}
\frac{(-1)^{2i}}{(4i+1)!}y^{4i+1}+\frac{(-1)^{2i+1}}{(4i+3)!}y^{4i+3}=\frac{(-1)^{2i}}{(4i+1)!}y^{4i+1}(1-\frac{y^2}{(4i+2)(4i+3)})
\end{align*}
Since $1-\frac{y^2}{(4i+2)(4i+3)}\geq 1-\frac{1}{6}>0$ the sign of the summation is determined by the sign of $y^{4i+1}$.

\end{proof}

\begin{proposition}\label{prop:cos_sign}
For any $y\in (-1,0)\cup (0,1)$ and $i\in \nat$ the following inequality holds:
\begin{align*}
\frac{(-1)^{2i}}{(4i)!}y^{4i}+\frac{(-1)^{2i+1}}{(4i+2)!}y^{4i+2}>0
\end{align*}
\end{proposition}

\begin{proof}
We can rewrite the left hand side of the inequality as follows:
\begin{align*}
\frac{(-1)^{2i}}{(4i)!}y^{4i}+\frac{(-1)^{2i+1}}{(4i+2)!}y^{4i+2}=\frac{(-1)^{2i}}{(4i)!}y^{4i}(1-\frac{1}{(4i+1)(4i+2)}y^2)
\end{align*}
Since $1-\frac{1}{(4i+1)(4i+2)}y^2\geq 1-\frac{1}{2}>0$, the inequality holds. 

\end{proof}

\begin{proposition}\label{prop:sin_condition}
For $x\in(-1,1)\setminus\{0\}$ we have $|x\cdot\cot(x)|<1$.
\end{proposition}

\begin{proof}
For the interval $x\in(-1,1)\setminus\{0\}$, we have $|x\cdot\cot(x)|=x\cdot\cot(x)$. We analyze the derivative of $f(x)=x\cdot\cot(x)$ in $(-1,1)\setminus\{0\}$.
\begin{align*}
f'(x)=\cot(x)-x(1+\cot^2(x))=\frac{\sin(x)\cos(x)-x}{\sin^2(x)}=\frac{\sin(2x)-2x}{2\sin^2(x)}
\end{align*}
The point $x=0$ is a critical point for $f(x)$. Since $|\sin(2x)|\leq |2x|$, $f(x)$ is increasing in $(-1,0)$ and decreasing in $(0,1)$. Thus, $f(x)$ is bounded from above by $\lim_{x\to 0}f(x)=1$. 

\end{proof}

\begin{proposition}\label{prop:cos_condition}
For $x\in(-1,1)$ we have $|x\cdot\tan(x)|<\tan(1)$.
\end{proposition}

\begin{proof}
For the interval $x\in(-1,1)$, we have $|x\cdot\tan(x)|=x\cdot\tan(x)$. To determine an upper-bound for $x\cdot\tan(x)$, we analyze the derivative of $f(x)=x\cdot\tan(x)$:
\begin{align*}
f'(x)=\tan(x)+x(1+\tan^2(x))=\frac{\sin(x)\cos(x)+x}{\cos^2(x)}=\frac{\sin(2x)+2x}{2\cos^2(x)}
\end{align*}
The quantity $\sin(2x)+2x$ is negative in $(-1,0)$ and positive in $(0,1)$. Thus, $f(x)$ is decreasing in $(-1,0)$ and increasing in $(0,1)$. We conclude that $f(x)<\tan(1)$.

\end{proof}

\begin{lemma}\label{lm:power}
Let $x$ be a real number represented by $(m,n,p)$. Then $x^i$ can be represented by $(m^i,n^i,p-2\lceil \log_2^i \rceil)$. 
\end{lemma}

\begin{proof}
We can apply Theorem~\ref{thm:mult} to calculate $x^i=x^{\lceil \frac{i}{2} \rceil}\times x^{\lfloor \frac{i}{2} \rfloor}$. Let $k=\lceil \log_2^i \rceil$ and $P(i)$ denote the precision that we lose by calculating $x^i$. From Theorem~\ref{thm:mult} we can write:
\begin{align*}
P(i)\leq P(\lceil \frac{i}{2} \rceil)+2\leq P(\lceil \frac{i}{2^k} \rceil)+2k=2k
\end{align*}
Thus, we lose $2k=2\lceil \log_2^i \rceil$ units of  precision by calculating $x^i$.

\end{proof}